\documentclass%[review]
{siamart220329}
\usepackage{braket,amsfonts}
\usepackage{amssymb}
\usepackage{mathrsfs}
\usepackage{array}
\usepackage{stmaryrd}
\usepackage{enumitem}
\usepackage{cases}
\newsiamthm{lem}{Lemma}
\newsiamthm{prop}{Proposition}
\newsiamthm{thm}{Theorem}
\newsiamremark{rmk}{Remark}
\newsiamthm{ass}{Assumption}
\newsiamthm{defn}{Definition}
\newsiamthm{pro}{Problem}
\newsiamthm{cor}{Corollary}

\renewcommand{\theequation}{\thesection.\arabic{equation}}
\makeatletter\@addtoreset{equation}{section} \makeatother

\allowdisplaybreaks
\begin{document}

	\title{Backward Stochastic Control System with Entropy Regularization\thanks{%Submitted to the editors %\today
		%\funding{
This paper is supported by National Key R\&D Program of China (No.2022YFA1006101), National Natural Science Foundation of China (No.12371445) and the Science and Technology Commission of Shanghai Municipality (No.22ZR1407600).}}%}
\author{Ziyue Chen
	\thanks{School of Mathematical Sciences, Fudan University, Shanghai 200433, China (\email{chenziyue21@m.fudan.edu.cn}).}
	\and Qi Zhang
	\thanks{Corresponding author. School of Mathematical Sciences, Fudan University, Shanghai 200433, China and Laboratory of Mathematics for Nonlinear Science, Fudan University, Shanghai 200433, China (\email{qzh@fudan.edu.cn}).}
}
		\maketitle
\begin{abstract}
The entropy regularization is inspired by information entropy from machine learning and the ideas of exploration and exploitation in reinforcement learning, which appears in the control problem to design an approximating algorithm for the optimal control. This paper is concerned with the optimal exploratory control for backward stochastic system, generated by the backward stochastic differential equation and with the entropy regularization in its cost functional. We give the theoretical depict of the optimal relaxed control so as to lay the foundation for the application of such a backward stochastic control system to mathematical finance and algorithm implementation. For this, we first establish the stochastic maximum principle by convex variation method. Then we prove sufficient condition for the optimal control and demonstrate the implicit form of optimal control. Finally, the existence and uniqueness of the optimal control for backward linear-quadratic control problem with entropy regularization is proved by decoupling techniques.
\end{abstract}
%\begin{AMS}
%93E20, 49K45, 49L20, 49N10
%\end{AMS}
\section{Introduction}
Different from the deterministic system which has only one path, there is much difference between the forward stochastic system and the backward one. It is well known that stochastic differential equation (SDE) and backward stochastic differential equation (BSDE) are much different from the form of equation to the form of solution. %Also (forward) stochastic partial differential equation (SPDE) and backward SPDE are much different from not only the form, but also the estimates of regularity.
In fact, BSDE and forward-backward stochastic differential equation (FBSDE) play a special role in many problems of stochastic analysis and stochastic controls. For example, Peng \cite{pe1991} demonstrates that the solution to FBSDE gives the probabilistic interpretation of nonlinear PDE which is known as nonlinear Feynman-Kac formula, Duffie and Epstein \cite{duffie-1992Econometrica} put forward the stochastic differential utility which is  actually a BSDE with conditional expectation, Zhang and Zhao \cite{zh-zh2007} constructs the stationary solution to parabolic stochastic partial differential equation (SPDE) based on the idea that infinite horizon backward doubly stochastic differential equation can serves as "elliptic" SPDE to give the pathwise steady statue of parabolic SPDE, to name but a few. Without the exception of control system, the forward stochastic control system and backward stochastic control system are also different. The difference not only lies on the state equation, i.e. the state equation of backward stochastic control system is controlled BSDE rather than SDE, but also on the application. It is well known that the controlled BSDE is widely used in mathematical finance, for example, as the dynamic equation for the value of portfolio to replicate a contingent claim. Also, as shown in Karnam, Ma and Zhang \cite{ka-ma-zh}, many  time-inconsistent optimization problems can be transformed into a stochastic controlled problem
with multidimensional BSDE dynamics by the so-called dynamic utility approach. Recently, the controlled BSDE is also applied to the numerical calculation of partial differential equation based on the theory of nonlinear Feynman-Kac formula. In an extended work of the deep BSDE numerical scheme by Takahashi, Tsuchida and Yamada \cite{ta-ts-ya}, the controlled BSDE is used to make the scheme more efficient and stable.

The state equation of backward stochastic control system is a controlled BSDE as below
\begin{equation}\label{equation:bsde1}
\left\{\begin{aligned}
-d y^{\pi}(t) & =f(t, y^{\pi}(t), z^{\pi}(t), \pi_t) d t-z^{\pi}(t) d W(t), \\
y(T) & =\xi,
\end{aligned}\right.
\end{equation}
where $\pi$ is the control variable. The studies on the backward stochastic control system emerged soon after the solvability of nonlinear BSDE.
Here we recall some early results. Peng \cite{peng-1993} first derived the local stochastic maximum principle in 1993. El Karoui, Peng and Quenez \cite{el-1997} demonstrated the application of controlled BSDE in finance. Dokuchaev and Zhou \cite{dokuchaev-1999} applied the backward stochastic control system to pricing European contingent claims and derived the global stochastic maximum principle in nonconvex control domain. For the linear-quadratic (LQ) case, it was studied in 2001 by Lim and Zhou \cite{lim-2001} with the help of decoupling techniques.

For $\sigma>0$, the first motivation to write this paper is to study the relaxed control of a type of backward stochastic control system with an entropy-regularized cost functional as below
\begin{equation}\label{eq:func bsde}
    J^{\sigma}(\pi)=\mathbb{E}\left[\int_0^T (l\left(t, y^\pi(t), z^\pi(t), \pi_t\right) + \frac{\sigma^2}{2} Ent(\pi_t \mid e^{-U}))d t+\phi\left(y^\pi(0)\right)\right].
\end{equation}
%The relaxed control originates from the machine learning and reinforcement learning.
The relaxed control means that the value of the control could depend on a distribution of value space, which actually enhances the possibility to get an optimal control, especially in a case that the classical control doesn't exist. This concept was put forward by Becker and Mandrekar \cite{be-ma1969} for deterministic control system in 1969 and extended to stochastic control system by Fleming \cite{fl1978} in 1978, and later El Karoui, H\.{u}\.{u}~Nguyen and Jeanblanc-Picqu\'{e} \cite{el-1987} further study the relaxed control for stochastic control system with degenerate diffusion. %Bahlali \cite{bahlali-2006} used relaxed control to prove the existence of optimal control for FBSDE without coupling $Z$.
It is worth noting that the relaxed control has a significant applications in machine learning algorithms, especially with the entropy regulation in the cost functional since the use of relaxed controls along with entropy regulation improves algorithm stability and efficiency. Actually, the entropy regulation has been widely used in numerical calculus for a long time. For example, the entropy regularization was ever applied to iterative numerical scheme for solving PDEs, and one can refer to Jordan, Kinderlehrer and Otto \cite{jordan-1998}, Gomes and Valdinoci \cite{gomes-2007} for details.
To apply this method to reinforcement learning, Wang, Zariphopoulou and Zhou \cite{wang-2020} studied the relaxed control with entropy-regularized cost functional to devise an exploratory formulation for the feature dynamics which captures learning under exploration based on the dynamic programming method. Moreover, in LQ case, \cite{wang-2020} proved that the optimal control is Gaussian. In the meantime, Wang and Zhou \cite{wangZhou-2020} showed that it has potential application in continuous mean-variance optimal portfolio problem. From the point of view of stability, Reisinger and Zhang \cite{reisinger-2021} demonstrated the regularised relaxed control formulation ensures that the optimal controls are stable with respect to model perturbations. A recent version of {\v S}i{\v s}ka and Szpruch \cite{siska-2020} added the priori reference measure into the entropy regularization and studied this general control system by the maximum principle method. Besides, there are more results on this topic for forward stochastic control system emerging recently, such as Gao, Xu and Zhou \cite{gao-2022},
Firoozi and Jaimungal \cite{firoozi-2022}, Jia and Zhou \cite{Jia-2023}, Dai, Dong, Jia and Zhou \cite{da-do-ji-zh}, etc. 

Unlike the control systems or research methods in \cite{wang-2020} and \cite{siska-2020}, we study the backward stochastic control systems \eqref{equation:bsde1} and \eqref{eq:func bsde} based on maximum principle method, which allows us to consider the control system with random coefficients. We get the necessary and sufficient condition for the optimal control which establishes a theoretical foundation for future applications on mathematical finance and machine learning. With the depict of the optimal control derived from the stochastic Hamiltonian system, it provides us a chance to give a specific implicit or explicit form of the relaxed optimal control. Actually, in the LQ case, we borrow the idea of decoupling techniques for backward stochastic control system in Lim and Zhou \cite{lim-2001} to prove the optimal control exactly exists and obtain its explicit from. As expected, the optimal relaxed control for backward stochastic LQ control systems with entropy regularization appears to be Gaussian. According to the theoretical depict of the optimal control for the backward stochastic control systems, we can also design an algorithm to approximate the optimal control by the method of successive approximation. Due to limited space, we will study it in another paper.

This paper is organized as follows. We introduce the necessary notation and state the backward stochastic control problem with entropy regularization in Section 2. Then we prove the extended maximum principle for our concerned control problem, and get the necessary condition of the optimal relaxed control in Section 3. The sufficient condition of the optimal relaxed control is given in Section 4, and the implicit form of optimal control is also discussed. In Section 5, we prove the existence and uniqueness of the optimal control for stochastic LQ control problem with entropy regularization and give the explicit form of optimal control in LQ case.

\section{Notation and Control Problem}%Main results
\label{sec:main}

Given a separable metric space $E$, let $\mathcal{M}(E)$ denote the set of all measures on $E$, $\mathcal{P}_q(E)$ denote the set of probability measures defined on $E$ with finite $q$-th moment for $q\in\mathbb{N}$ and $\mathcal{P}(E)=\mathcal{P}_0(E)$ denote the set of all probability measures on $E$. In this paper the entropy of the measure $\pi \in \mathcal{P}(E)$ is defined as below
\begin{equation*}\label{eq: entropy}
    R(\pi) = \left\{\begin{array}{l}
        - \int_E \frac{d\pi}{du}\ln\frac{d\pi}{du} du, \quad \text{if}\ \pi \ll \lambda(du), \\
          -\infty, \quad \text{otherwise,}
    \end{array}\right.
\end{equation*}
where $\lambda$ is the Lebesgue measure on a separable metric space $E$. In order to make above entropy well defined, we assume that all probability-measure-valued control in this paper are absolutely continuous with respect to $\lambda$. Hence by Radon-Nikodym theorem, there exists a measurable function $g:U \rightarrow \mathbb{R}^+$ such that $\mu(du) = g(u)du$ a.e. Throughout the paper we will abuse the notation of probability measure which are absolutely continuous with respect to Lebesgue measure and do not distinguish a probability measure from its density which exists due to Radon-Nikodym theorem.

We begin with a finite time horizon $[0, T]$ for $T>0$ and a complete filtered probability space $\left(\Omega, \mathscr{F}, \mathbb{P}\right)$, on which a standard $\mathbb{R}^{m}$-valued Brownian motion $W$ is defined. Moreover, $\mathbb{F}=\left(\mathscr{F}_t\right)_{0 \leq t \leq T}$ is the natural filtration generated by $W$ and $\mathscr{F}_T=\mathscr{F}$. %Unlike to forward systems, we refer to Huang and Xiong\cite{huang-2009} for backward systems with full information.
Next we introduce some useful spaces. For $t\in[0,T]$ and a Hilbert space $S$ with norm $\|\cdot\|_S$
and Borel $\sigma$-field $\mathscr{S}$, we define\\
$\bullet$~~$S_{\mathscr{F}}^2(0, T ; S)$: the space of all $\mathbb{F}$-adapted processes $x:\Omega\times[0,T]\rightarrow S$ satisfying $t\rightarrow x(t)$ is a.s. continuous and $\mathbb{E} \left[\sup_{0\leq t\leq T}\|x(t)\|_S^2 \right]<+\infty$;\\
$\bullet$~~$L_{\mathscr{F}}^2(0, T ; S)$: the space of all $\mathbb{F}$-adapted processes $x:\Omega\times[0,T]\rightarrow S$ satisfying $\mathbb{E} \left[\int_0^T\|x(t)\|_S^2dt \right]<+\infty$;\\
$\bullet$~~$L_{\mathscr{F}_t}^2(\Omega ; S)$: the space of all $\mathscr{F}_t$-measurable random variables $x:\Omega\rightarrow S$ satisfying $\mathbb{E}\left[\|x(t)\|_S^2\right]<+\infty$;\\
$\bullet$~~$L^{\infty}(0,T;S)$: the space of all measurable maps $x: [0, T] \rightarrow S$ satisfying $\|x\|_{L^{\infty}}=\operatorname{ess} \sup _{t \in[0, T]}\Vert x(t) \Vert_S<\infty$;\\
$\bullet$~~$L^{\infty}(S)$: the space of all measurable maps $x: S \rightarrow \mathbb{R}$ satisfying $\|x\|_{L^{\infty}(S)}= \operatorname{ess} \sup_{a \in S} |x(a)|<\infty$;\\
$\bullet$~~$C^{\infty}$: the space of all infinitely differentiable functions.

To avoid heavy notations, we omit the transpose symbols in this paper unless necessary.

For the backward stochastic control systems \eqref{equation:bsde1} and \eqref{eq:func bsde}, the set of admissible controls $\mathcal{A}$ is defined as follows.
\begin{definition}\label{def:admissible control1}
    For $E=[0, T] \times \mathbb{R}^p$, we define a subset of $\mathcal{M}(E)$:
    {\small$$
    \begin{aligned}
    & \mathcal{M}_2:=\Big\{\pi \in \mathcal{M}(E): \text { for a.a. } t \in[0, T],  \text { there exists }\ \pi_t \in \mathcal{P}(\mathbb{R}^p) \text{ such that}\\
    &\ \quad \quad \quad \quad \quad \quad \quad \quad \quad \pi(d a, d t)=\pi_t(a) d a d t,\ \int_0^T \int|a|^2 \pi_t(a) d a d t <\infty \Big\}.
    \end{aligned}
    $$}
Here and in the rest of this paper, the integration without explicit domain is over $\mathbb{R}^p$ unless indicated explicitly. Then the set of admissible controls
{\small$$
\begin{aligned}
& \mathcal{A}:=\Big\{\pi: \Omega \rightarrow \mathcal{M}_2: \mathbb{E}\left[ \int_0^T Ent(\pi_t\mid e^{-U}) d a d t\right]<\infty,\ \mathbb{E}\left[\int_0^T \int|a|^2 \pi_t(a) d a d t\right]<\infty\\
&\ \quad \quad \quad \quad \quad \quad \quad \quad \quad \text { and } \pi_t \text { is } \mathscr{F}_t \text {-measurable for any}\ t \in[0, T]\Big\}.
\end{aligned}
$$}
where $U:\mathbb{R}^p\to\mathbb{R}$ is a measurable function such that $e^{-U}$ is a density function. Here $e^{-U}$ is regarded as the priori reference measure and $Ent(\cdot \mid \cdot)$ is the relative entropy characterized by
$$
Ent(m \mid m^{\prime}) := \int \left(\ln m(a) - \ln m^{\prime}(a)\right) m(a) da,
$$
for $m,m^{\prime} \in \mathcal{P}(\mathbb{R}^p)$ which are absolutely continuous with respect to the Lebesgue measure (otherwise $Ent(m \mid m^{\prime}) = \infty$).
\end{definition}

\begin{rmk}
By Lemma A.1 in Kerimkulov, {\v S}i{\v s}ka, Szpruch and Zhang \cite{kerimkulov-2024}, the admissible control set $\mathcal{A}$ is convex due to the fact that $\pi \mapsto Ent(\pi \mid e^{-U})$ is convex.
\end{rmk}

For the given measurable functions $f: \Omega\times[0,T] \times \mathbb{R}^n \times \mathbb{R}^{n \times m} \times \mathcal{P}(\mathbb{R}^p) \rightarrow \mathbb{R}^n$, $l: \Omega\times[0,T] \times \mathbb{R}^n \times \mathbb{R}^{n \times m}  \times \mathcal{P}(\mathbb{R}^p) \rightarrow \mathbb{R}$, $\xi: \Omega\rightarrow\mathbb{R}^n$, $\phi:\mathbb{R}^n \rightarrow \mathbb{R}^n$ and the admissible control $\pi \in \mathcal{A}$, we consider the backward stochastic system \eqref{equation:bsde1} and \eqref{eq:func bsde}. Then the control problem is

(\textbf{P1}): to find an optimal $\mu \in \mathcal{A}$ such that
$$
J^{\sigma}(\mu) = \inf_{\pi \in \mathcal{A} } J^{\sigma}(\pi).
$$

Since the concerned system involves the measure-valued control, we need define flat derivative on a convex subset $\mathcal{C} \subseteq \mathcal{P}\left(\mathbb{R}^p\right)$. The following definition refers to \cite{kerimkulov-2024} based on the ideas from early literatures (e.g. Lions \cite{li}, Carmona and Delarue \cite{ca-de}, Buckdahn, Li, Peng and Rainer \cite{bu-li-pe-ra}, etc.).

\begin{definition}\label{def: flat derivative}
    A functional $F: \mathcal{C} \rightarrow \mathbb{R}^d$ is said to admit a linear derivative if there is a continuous map $\frac{\delta F}{\delta m}: \mathcal{C} \times \mathbb{R}^p \rightarrow \mathbb{R}^d$ such that for all $m, m^{\prime} \in \mathcal{C}$, it holds that $\int \left|\frac{\delta F}{\delta m}( m)(a)\right| m^{\prime}(a) d(a) < \infty$, and
    \begin{equation}\label{eq: flat derivative01 }
        F\left(m^{\prime}\right)-F(m)=\int_0^1 \int \frac{\delta F}{\delta m}\left(m+\lambda\left(m^{\prime}-m\right)\right)(a)  \cdot \left(m^{\prime}(a)-m(a)\right)da d \lambda .
    \end{equation}
\end{definition}

In the above definition, $\frac{\delta F}{\delta m}$ is only defined up to a constant according to Remark 5.46 in Carmona and Delarue \cite{carmona-2018}. The functional $\frac{\delta F}{\delta m}$ is then called the linear (functional) derivative of $F$ on $\mathcal{C}$. Note that if $\frac{\delta F}{\delta m}$ exists, %according to Definition \ref{def: flat derivative} then
for any $\nu, \mu \in \mathcal{C}$,
\begin{equation}\label{eq: flat derivative02 }
    \lim _{\varepsilon \rightarrow 0+} \frac{F(\nu+\varepsilon(\mu-\nu))-F(\nu)}{\varepsilon}=\int \frac{\delta F}{\delta m}(\nu)(a)  (\mu(a)-\nu(a))da .
\end{equation}
Obviously, (\ref{eq: flat derivative01 }) implies (\ref{eq: flat derivative02 }). To see the implication in the other direction, take $\nu^\lambda:=\nu+\lambda(\mu-\nu)$ and $\mu^\lambda:=\mu-\nu+\nu^\lambda$ and notice that (\ref{eq: flat derivative02 }) ensures that for all $\lambda \in[0,1]$,
$$
\begin{gathered}
\lim _{\varepsilon \rightarrow 0+} \frac{F\left(\nu^\lambda+\varepsilon(\mu-\nu)\right)-F\left(\nu^\lambda\right)}{\varepsilon}=\lim _{\varepsilon \rightarrow 0+} \frac{F\left(\nu^\lambda+\varepsilon\left(\mu^\lambda-\nu^\lambda\right)\right)-F\left(\nu^\lambda\right)}{\varepsilon} \\
=\int \frac{\delta F}{\delta m}\left(\nu^\lambda\right)(a)  \left(\mu^\lambda-\nu^\lambda\right)da=\int \frac{\delta F}{\delta m}  \left(\nu^\lambda\right)(a) (\mu(a)-\nu(a))da .
\end{gathered}
$$
By the fundamental theorem of calculus, we can derive that
\begin{equation}\label{eq:A.2}%{eq: prop of flat dev}
  F(\mu)-F(\nu)=\int_0^1 \lim _{\varepsilon \rightarrow 0+} \frac{F\left(\nu^{\lambda+\varepsilon}\right)-F\left(\nu^\lambda\right)}{\varepsilon} d \lambda=\int_0^1 \int \frac{\delta F}{\delta m} \left(\nu^\lambda\right)(a) (\mu(a)-\nu(a))da d \lambda.
\end{equation}

The linear derivative $\frac{\delta F}{\delta \nu}$ is here also defined up to the additive constant as any constant can be added to $\int \frac{\delta F}{\delta \nu}(\nu, t)(a) \nu_t(d a) $ without affecting the definition formula. % (Remark 5.46 in \cite{carmona-2018}).
Note that if $\frac{\delta F}{\delta \nu}$ exists, 
then similarly we have an equivalent form of (\ref{eq:A.2}) as below
\begin{equation*}\label{eq:A.3}
    \forall \nu, \nu^{\prime} \in \mathcal{A}, \lim _{\epsilon \rightarrow 0^{+}} \frac{F\left(\nu+\epsilon\left(\nu^{\prime}-\nu\right)\right)-F(\nu)}{\epsilon}= \int \frac{\delta F}{\delta \nu}(\nu, t)(a)\left(\nu_t^{\prime}(a)-\nu_t(a)\right)d a.
\end{equation*}

Then we state the assumptions in our concerned control problem. 
\begin{ass}\label{H: bsde}
(i) $\xi \in L^2_{\mathscr{F}_T}(\Omega; \mathbb{R}^n)$. Denote by $\mathscr{P}$ the predictable sub-$\sigma$ algebra of $\mathscr{F} \otimes \mathcal{B}([0,T])$, and $f$ is $\mathscr{P} \otimes \mathcal{B}(\mathbb{R}^n) \otimes \mathcal{B}(\mathbb{R}^{n \times m})  \otimes \mathcal{B}(\mathcal{P}(\mathbb{R}^p))$-measurable and there exists a $m \in \mathcal{P}(\mathbb{R}^p)$ such that $f(\cdot, 0,0, m) \in L^2_{\mathscr{F}}(0,T;\mathbb{R}^n)$. Also for $(\omega,t)\in\Omega\times[0,T]$, $f(t,y,z,m)$ is continuously differentiable with respect to $(y,z)\in\mathbb{R}^n\times\mathbb{R}^{n \times m}$ and continuously twice differentiable with respect to $m\in\mathcal{P}_2(\mathbb{R}^p)$ (in the sense of flat derivative), and for any $(\omega, t) \in \Omega \times [0, T]$, $(y,z)$, $(y^{\prime},z^{\prime})\in\mathbb{R}^n\times\mathbb{R}^{n \times m}$, $m\in\mathcal{P}_2(\mathbb{R}^p)$, there exists a constant $K > 0$ such that $\left| \nabla_y f \right| + \left| \nabla_z f \right| +\vert \frac{\delta^2 f}{\delta m^2}\vert\leq K$ uniformly and
    $$
    \begin{aligned}
            \left| \frac{\delta f(t,y,z,m)}{\delta m} - \frac{\delta f(t,y^{\prime},z^{\prime},m)}{\delta m}\right| & \leq K(\left| y - y^{\prime} \right| + \left| z - z^{\prime} \right| ) ;\\
            \left| \frac{\delta f(t,y,z,m)}{\delta m} \right| & \leq K(1 + \left|a\right|).
    \end{aligned}
    $$
    %implying the linear growth of $\frac{\delta f}{\delta m}$.
    (ii) $\phi$ is $\mathcal{B}(\mathbb{R}^n)$-measurable and $l$ is $\mathscr{P} \otimes \mathcal{B}(\mathbb{R}^n) \otimes \mathcal{B}(\mathbb{R}^{n \times m}) \otimes \mathcal{B}(\mathcal{P}(\mathbb{R}^p))$-measurable. For $(\omega,t)\in\Omega\times[0,T]$, $\phi(y)$ is continuously differentiable with respect to $y\in\mathbb{R}^n$, and $l(t,y,z,m)$ is continuously differentiable with respect to $(y,z)\in\mathbb{R}^n\times\mathbb{R}^{n \times m} $ and continuously twice differentiable with respect to $m\in\mathcal{P}_2(\mathbb{R}^p)$, and for any $(\omega, t) \in \Omega \times [0, T]$, $(y,z)$, $(y^{\prime},z^{\prime})\in\mathbb{R}^n\times\mathbb{R}^{n \times m}$, $m\in\mathcal{P}_2(\mathbb{R}^p)$ and $K$ in (i),
            $$
        \begin{aligned}
        & |\phi| \leq K\left(1+|y|^2\right), \\
        & |l| \leq K\left(1+|y|^2+|z|^2\right),\\
        & \left|\nabla_y \phi\right| \leq K(1+|y|),\\
        &\left|\nabla_y l\right|+\left|\nabla_z l\right|\leq K(1+|y|+|z|), \\
        & \left| \frac{\delta l(t,y,z,m)}{\delta m} - \frac{\delta l(t,y^{\prime},z^{\prime},m)}{\delta m}\right| \leq K(\left| y - y^{\prime} \right| + \left| z - z^{\prime} \right| ),\\
        %&{\color{red} \xout{|\frac{\delta l}{\delta %m} |\leq K(1+|m|),}} \\
        &\left| \frac{\delta l(t,y,z,m)}{\delta m} \right|  \leq K(1 + \left|a\right|^2);\\
        & |\frac{\delta^{2} l}{\delta m^{2}} |\leq K.\\
        \end{aligned}
        $$
\end{ass}

From Assumption \ref{H: bsde} (i), we know that $f$ is a uniformly Lipschitz generator. Hence for any $\pi \in \mathcal{A}$, by Pardoux and Peng \cite{par-pen1} the state equation (\ref{equation:bsde1}) has a unique solution $(y^\pi, z^\pi) \in S_{\mathscr{F}}^2\left(0, T ; \mathbb{R}^n\right) \times L_{\mathscr{F}}^2\left(0, T ; \mathbb{R}^{n\times m}\right)$.

From Assumption \ref{H: bsde} (i) (ii), we know that the optimization of control problem is well-posed since
$$
\mathbb{E}\left[\left|\phi\left(y^\pi_0\right)\right|\right]<\infty\ \ \ {\rm and}\ \ \ \mathbb{E}\left[ \int_0^T\left|l\left(t, y^\pi_t, z^\pi_t, \pi_t \right)\right| d t\right]<\infty,
$$
together with the fact that $\mathbb{E}\left[ \int_0^T \int Ent(\pi_t \mid e^{-U(a)}) da \right]< \infty$.

\section{Maximum Principle}
   % \subsubsection{Variation Estimates}
    In this section, we denote by $(Y^\pi, Z^\pi)$ the solution to BSDE (\ref{equation:bsde1}) driven by $\pi \in \mathcal{A}$. Since the admissible control set $\mathcal{A}$ is convex, we will work with an additional control $\pi \in \mathcal{A}$ and define $\mu^{\epsilon} := \mu + \epsilon (\pi - \mu)$. Denote by $(Y^{\epsilon}, Z^{\epsilon})$ the solution to BSDE (\ref{equation:bsde1}) driven by $\mu^{\epsilon} \in \mathcal{A}$. Firstly, we have some prior estimates.
    \begin{lemma}\label{lem:big order estimate}
     Under Assumption \ref{H: bsde}, then
     \begin{equation*}\label{eq:higher order estimate}
         \mathbb{E}\left[\sup_{t \in [0,T]} |Y^{\epsilon}_t - Y^\mu_t|^2\right] + \mathbb{E}\left[\int_{0}^{T} |Z^{\epsilon}_t - Z^\mu_t|^2dt\right] = O(\epsilon^2).
     \end{equation*}
    \end{lemma}
    \begin{proof}
     By Assumption \ref{H: bsde}, $\frac{\delta f}{\delta m}$  is linear growth on $a$. Hence 
     a classical priori estimates of BSDE leads to
        \begin{align*}
             &\mathbb{E} \left[\sup_{t \in [0,T]} |Y^{\epsilon}_t - Y^\mu_t|^2 \right] + \mathbb{E} \left[\int_{0}^{T} |Z^{\epsilon}_t - Z^\mu_t|^2 dt \right] \\
            & \leq C \mathbb{E}\left[\int_0^T \left|f(t,Y^{\epsilon}_t, Z^{\epsilon}_t, \mu^{\epsilon}_t )- f(t,Y^{\epsilon}_t, Z^{\epsilon}_t,  \mu_t )\right|^2dt\right] \\
            & = C  \mathbb{E}\left[\int_0^T \left| \int_0^1 \int \frac{\delta f}{\delta m}(t,Y^{\epsilon}_t, Z^{\epsilon}_t, (1-\lambda)\mu_t + \lambda \mu^{\epsilon}_t)(a)\epsilon(\pi_t(a) - \mu_t(a))dad\lambda \right|^2  dt\right]\\
            & \leq C  \mathbb{E}\left[\int_0^T \left| \epsilon K\int_0^1 \int (1+|a|)(\pi_t(a) + \mu_t(a))dad\lambda \right|^2 dt\right]\\
           & \leq C\epsilon^2 K^2 \mathbb{E} \left[\int_0^T \int_0^1 \int (1+|a|)^2(\pi_t(a) + \mu_t(a))dad\lambda dt \right]\\
            & \leq  C_{K,T} \epsilon^2= O(\epsilon^2).
        \end{align*}
The first equality is due the Definition \ref{def: flat derivative}. 
       Here and in the rest of this paper, $C$ is a generic constant whose value may change line by line, and the subscript of $C$, if it is indicated,  is the given parameters $C$ depends on.
    \end{proof}

    For $\psi = f,l$; $x = y,z$, set
  \begin{equation}\label{eq: difference eq}
        \left\{\begin{aligned}
            \nabla_x \psi(t) & = \nabla_x \psi(t, Y^\mu_t, Z^\mu_t, \pi_t ),\\
            \nabla_x\tilde{\psi}^{\epsilon}(t) & = \int_0^1 \nabla_x \psi(t, Y^\mu_t + \lambda (Y^{\epsilon}_t - Y^\mu_t), Z^\mu_t + \lambda (Z^{\epsilon}_t - Z^\mu_t), \pi_t)d\lambda.
        \end{aligned}\right.
        \nonumber
    \end{equation}
By the uniform boundeness of  $\left| \nabla_x \tilde{f}^{\epsilon} \right|$ and the linear growth of $\left| \frac{\delta f}{\delta m}\right|$, BSDE
   {\small \begin{equation}
        \left\{\begin{aligned}
    -d (Y^{\epsilon}_t - Y^\mu_t) & =-\Big( \nabla_y\tilde{f}^{\epsilon}(t)(Y^{\epsilon}_t - Y^\mu_t) + \nabla_z\tilde{f}^{\epsilon}(t)
            (Z^{\epsilon}_t - Z^\mu_t) \\
            &\quad\quad +\epsilon \int_0^1 \int \frac{\delta f}{\delta m}(t,Y^{\epsilon}_t, Z^{\epsilon}_t, (1-\lambda)\mu_t + \lambda \mu^{\epsilon}_t)(a)(\pi_t(a) - \mu_t(a))da d\lambda \Big)dt\\
            &\quad + (Z^{\epsilon}_t - Z^\mu_t)dW_t, \\
    Y^{\epsilon}_T - Y^\mu_T =0,
    \end{aligned}\right.
    \end{equation}}
has a unique solution.
    Next, we introduce the variation equation
    \begin{equation}\label{eq:variation eq}
        \left\{\begin{aligned}
            d V_t & = -\Big(\nabla_y f(t) V_t + \nabla_z f(t) Z^{V}_t \Big)dt + Z^{V}_t dW_t \\
            & \quad +  \int \frac{\delta f}{\delta m}(t,Y^\mu_t, Z^\mu_t, \mu_t)(a)(\pi_t(a) - \mu_t(a))da  dt ,\\
            V_T & = 0.
        \end{aligned}\right.
    \end{equation}
Similarly, by the boundeness of $\left| \nabla_x f \right| $ and  the linear growth of $\left| \frac{\delta f}{\delta m}\right|$, the variation equation (\ref{eq:variation eq}) has a unique solution. Moreover, we can get the following estimate
    \begin{equation}\label{eq:variation eq's estimate}
        \begin{aligned}
                &\mathbb{E}\left[\sup_{t \in [0,T]} |V_t|^2 \right] + \mathbb{E}\left[\int_{0}^{T} |Z^{V}_t|^2 dt\right] \notag\\
                &\leq C \mathbb{E}\left[\int_0^T  \left|\int \frac{\delta f}{\delta m}(t,Y^\mu_t, Z^\mu_t, \mu_t)(a)(\pi_t(a) - \mu_t(a))da \right|^2dt\right]\notag\\
                & \leq C K^2 \mathbb{E} \left[\int_0^T \int_0^1 \int (1+|a|)^2(\pi_t(a) + \mu_t(a))dad\lambda dt \right]\notag\\
                &\leq 2C_{K,T} .
    \end{aligned}
    \end{equation}

    Set $V^{\epsilon}_t := Y^{\epsilon}_t - Y^\mu_t - \epsilon V_t$ and $Z^{V^{\epsilon}}_t := Z^{\epsilon}_t - Z^\mu_t - \epsilon Z^{V}_t$. % and $\bar{Z}^{V^{\epsilon}}_t := \bar{Z}^{\epsilon}_t - \bar{Z}^\mu_t - \epsilon \bar{Z}^{V}_t$.
    Then we have more estimates.
    \begin{lemma}\label{lem: small order estimate }
        Under Assumption \ref{H: bsde},
        \begin{equation*}
                    \begin{aligned}
                        & \mathbb{E}\left[\sup_{t \in [0,T]} |V^{\epsilon}_t|^2 \right] + \mathbb{E} \left[\int_{0}^{T} (|Z^{V^{\epsilon}}_t|^2 ) dt \right] = o(\epsilon^2).
                    \end{aligned}
        \end{equation*}
    \end{lemma}
    \begin{proof}First note that BSDE
        \begin{equation}\label{eq:epsilon variation eq}
        \left\{\begin{aligned}
            d \epsilon V_t & = -\epsilon \Big(\nabla_y f(t)  V_t + \nabla_z f(t)Z^{V}_t \Big )dt + \epsilon Z^{V}_t dW_t \\
            & \quad + \epsilon \int_0^1 \int \frac{\delta f}{\delta m}(t,Y^\mu_t, Z^\mu_t, \mu_t)(a)(\pi_t(a) - \mu_t(a))dad\lambda dt ,\\
            \epsilon V_T & = 0,
        \end{aligned}\right.
    \end{equation}
        has a unique solution. Based on BSDEs (\ref{eq: difference eq}) and (\ref{eq:epsilon variation eq}), by the continuity dependence of the solutions to BSDEs, we have
       $$
        \begin{aligned}
            & \mathbb{E} \left[\sup_{t \in [0,T]} |V^{\epsilon}_t|^2 \right] + \mathbb{E} \left[\int_{0}^{T} |Z^{V^{\epsilon}}_t|^2 dt \right] \leq C \epsilon^2 h(\epsilon),
        \end{aligned}
        $$
        where
      $$
       \begin{aligned}
            &h(\epsilon):= \mathbb{E}\bigg[\int_0^T \bigg(\Big| \left(\nabla_y\tilde{f}^{\epsilon}(t) - \nabla_y f(t)\right)  V_t + \left(\nabla_z\tilde{f}^{\epsilon}(t) - \nabla_z f(t)\right) Z^{V}_t\\
            & \quad \quad\quad\quad\quad\quad\ \ \ + \int_0^1 \int \left(\frac{\delta f}{\delta m}(t,Y^{\epsilon}_t, Z^{\epsilon}_t, (1-\lambda)\mu_t + \lambda \mu^{\epsilon}_t)(a)\right.\\
            & \quad \quad \quad \quad \quad\quad\quad\quad\quad \quad \ \ \ \ \ \left.- \frac{\delta f}{\delta m}(t,Y^\mu_t, Z^\mu_t, \mu_t)(a)\right) \left(\pi_t(a) - \mu_t(a) \right) da d\lambda  \Big|^2 \bigg)d t\bigg].
        \end{aligned}
        $$
        It is sufficient to show that $\lim_{\epsilon \rightarrow 0^+} h(\epsilon) = 0$. Note that
            \begin{equation}\label{eq:h(epsilon)}
         \begin{aligned}
            h(\epsilon) & \leq 2 \mathbb{E}\bigg[\int_0^T \Big| \left(\nabla_y\tilde{f}^{\epsilon}(t) - \nabla_y f(t)\right)  V_t+ \left(\nabla_z\tilde{f}^{\epsilon}(t) - \nabla_z f(t)\right) Z^{V}_t \Big|^2 d t\bigg]\\
            & \quad + 2\mathbb{E}\Big[\int_0^T \Big|\int_0^1 \int \left(\frac{\delta f}{\delta m}(t,Y^{\epsilon}_t, Z^{\epsilon}_t, (1-\lambda)\mu_t + \lambda \mu^{\epsilon}_t)(a) \right. \\
            & \quad \quad \quad \quad \quad \quad \quad \quad \quad \quad \left.- \frac{\delta f}{\delta m}(t,Y^\mu_t, Z^\mu_t, \mu_t)(a)\right) \left(\pi_t(a) - \mu_t(a)\right) da d\lambda\Big|^2 dt \Big] \\
            & \leq 2\mathbb{E}\left[\int_0^T\left(h_1(t,\epsilon) + h_2(t,\epsilon)\right)dt\right],
        \end{aligned}
        \end{equation}
where $$
        \left\{\begin{aligned}
             h_1(t,\epsilon)  &= \left|\left(\nabla_y\tilde{f}^{\epsilon}(t) - \nabla_y f(t)\right)^2  + \left(\nabla_z\tilde{f}^{\epsilon}(t) - \nabla_z f(t)\right)^2  \right|\left(|V_t|^2 + |Z^{V}_t|^2\right),\\
             h_2(t,\epsilon)
             &= \Big|\int_0^1 \int \left(\frac{\delta f}{\delta m}(t,Y^{\epsilon}_t, Z^{\epsilon}_t, (1-\lambda)\mu_t + \lambda \mu^{\epsilon}_t)(a) \right. \\
            & \quad \quad \quad \quad \quad \ \ \left.-\frac{\delta f}{\delta m}(t,Y^\mu_t, Z^\mu_t, \mu_t)(a)\right)\left(\pi_t(a) - \mu_t(a)\right) da d\lambda \Big|^2.
        \end{aligned}\right.
        $$
Set $$
        \left\{\begin{aligned}
             h_1(t)  &= 8K^2 (|V_t|^2 + |Z^{V}_t|^2),\\
             h_2(t) &= 32K^2 + 16K^2 \left(|Y^{\epsilon}_t - Y_t|^2 +  |Z^{\epsilon}_t - Z_t|^2 \right).
        \end{aligned}\right.
        $$
        By Assumption \ref{H: bsde} and regularity of probability density, we further have for $i \in \{1,2\}$ and any $0 < \epsilon < 1$,
        \begin{equation}\label{eq: dominate}
            h_i(t,\epsilon) \leq h_i(t)\ \ {\rm a.e.} \ {\rm a.s.}
        \end{equation}
Actually, by a priori estimate (\ref{eq:variation eq's estimate}) and Lemma \ref{lem:big order estimate}, for $i = 1,2$,
        \begin{equation}\label{eq: dominate integrable}
            \mathbb{E}\left[\int_0^T h_i(t,\epsilon) dt\right] < \infty.
        \end{equation}

        As $\epsilon_n \downarrow 0$, by definition of $\mu_t^{\epsilon}$, we know that $\mu_t^{\epsilon_n}$ converges weakly %weakly\cite{villani-2013}
        to $\mu_t$ a.e. a.s. Also by Lemma \ref{lem:big order estimate} it is clear that $(Y^{\epsilon_n}_t, Z^{\epsilon_n}_t) $ converges in $S_{\mathscr{F}}^2\left(0, T ; \mathbb{R}^n\right) \times L_{\mathscr{F}}^2\left(0, T ; \mathbb{R}^{n\times m}\right)$, so there exists a subsequence of $\{\epsilon_n\}$, still denoted by $\{\epsilon_n\}$, such that $(Y^{\epsilon_n}_t, Z^{\epsilon_n}_t) $ converges to $(Y^\mu_t, Z^\mu_t)$ a.e. a.s. By the continuity of $\nabla_y f$ with respect to $(y,z)$, we have for $0 \leq \lambda \leq 1$,
      $$
        \begin{aligned}
        \lim_{n \rightarrow \infty} \left|\nabla_y f(t, Y^\mu_t + \lambda (Y^{\epsilon_n}_t - Y^\mu_t), Z_t + \lambda (Z^{\epsilon_n}_t - Z^\mu_t), \mu_t) - \nabla_y f(t) \right|= 0. %\ {a.e.} \ {a.s.}
        \end{aligned}
        $$
Hence by the dominated convergence theorem,
        % small footsize
        $$
        %\begin{small}
        \begin{aligned}
           & \lim_{n \rightarrow \infty} |\nabla_y\tilde{f}^{\epsilon_n}(t) - \nabla_y f(t)|^2\\
            & \leq \lim_{n \rightarrow \infty} \int_0^1 |\nabla_y f(t, Y^\mu_t + \lambda (Y^{\epsilon_n}_t - Y^\mu_t), Z^\mu_t + \lambda (Z^{\epsilon_n}_t - Z^\mu_t), \mu_t) - \nabla_y f(t)|^2 d \lambda\\
            & =0\ \ {\rm a.e.}\ {\rm a.s.}
        \end{aligned}
        %\end{small}
        $$
        Similarly, $\lim_{n \rightarrow \infty} |\nabla_z\tilde{f}^{\epsilon_n}(t) - \nabla_z f(t)|^2 = 0$.

        For the term of control, we notice that
        %\begin{small}
            {\small$$\begin{aligned}
            &\int \Big(\frac{\delta f}{\delta m}(t,Y^{\epsilon_n}_t, Z^{\epsilon_n}_t, (1-\lambda)\mu_t + \lambda \mu^{\epsilon_n}_t)(a) - \frac{\delta f}{\delta m}(t,Y^\mu_t, Z^\mu_t, \mu_t)(a)\Big)(\pi_t(a) - \mu_t(a)) da  \\
            & = \int \Big(\frac{\delta f}{\delta m}(t,Y^{\epsilon_n}_t, Z^{\epsilon_n}_t, (1-\lambda)\mu_t + \lambda \mu^{\epsilon_n}_t)(a) - \frac{\delta f}{\delta m}(t,Y^{\epsilon_n}_t, Z^{\epsilon_n}_t, \mu_t)(a)\Big)(\pi_t(a) - \mu_t(a)) da\\
            & \quad + \int \Big(\frac{\delta f}{\delta m}(t,Y^{\epsilon_n}_t, Z^{\epsilon_n}_t,  \mu_t)(a) - \frac{\delta f}{\delta m}(t,Y^\mu_t, Z^\mu_t, \mu_t)(a)\Big)(\pi_t(a) - \mu_t(a)) da.
        \end{aligned}$$}
        %\end{small}
For $0\leq\lambda^{\prime}\leq1$ and $\mu_t^{\lambda,\lambda^{\prime}} := (1-\lambda^{\prime})\mu_t + \lambda^{\prime}((1-\lambda)\mu_t + \lambda \mu_t^{\epsilon_n}) = \mu_t + \lambda \lambda^{\prime}(\mu_t^{\epsilon_n} - \mu_t)$,
      \begin{equation}\label{eq: estimate f2}
            \begin{aligned}
            &\Big|\mathbb{E} \Big[ \int_0^T \int_0^1 \int \Big(\frac{\delta f}{\delta m}(t,Y^{\epsilon_n}_t, Z^{\epsilon_n}_t, (1-\lambda)\mu_t + \lambda \mu^{\epsilon_n}_t)(a)\\
            & \quad \quad \quad \quad \quad \quad \quad -  \frac{\delta f}{\delta m}(t,Y^{\epsilon_n}_t, Z^{\epsilon_n}_t, \mu_t )(a)\Big) \epsilon_n(\pi_t(a)-\mu_t(a)) da d \lambda dt\Big] \Big|^2\\
            &  \leq \mathbb{E}\Big[  \int_0^T \Big| \int_0^1 \int_0^1 \lambda \int \int \frac{\delta^2 f}{\delta m^2}(t,Y^{\epsilon_n}_t, Z^{\epsilon_n}_t, \mu_t^{\lambda,\lambda^{\prime}})(a,a^{\prime}) \epsilon_n(\pi_t(a^{\prime})-\mu_t(a^{\prime})) d a^{\prime}\\
            & \quad \quad \quad \quad \quad \quad \quad \quad \quad \quad  \quad \cdot\epsilon_n(\pi_t(a)-\mu_t(a)) da d \lambda d \lambda^{\prime}\Big|^2 dt\Big]\\
            & \leq \mathbb{E} \Big[ \int_0^T \Big| \int_0^1 \int_0^1 \lambda K \int (\pi_t(a^{\prime})+\mu_t(a^{\prime})) d a^{\prime}\int(\pi_t(a)+\mu_t(a)) da d \lambda d \lambda^{\prime}\Big|^2 dt\Big] \epsilon_n^2\\
            & =  C_{K,T}\epsilon_n^2= o(\epsilon_n).
            \end{aligned}
        \end{equation}
Then, since $\frac{\delta f}{\delta m}$ is uniformly Lipschitz in $(y,z)$ for any fixed $m \in \mathcal{A}$, by Lemma \ref{lem:big order estimate} we get
        {\small\begin{align}\label{eq: estimate f3}
            &\Big|\mathbb{E} \Big[ \int_0^T \int_0^1 \int \Big(\frac{\delta f}{\delta m}(t,Y^{\epsilon_n}_t, Z^{\epsilon_n}_t, \mu_t)(a)-  \frac{\delta f}{\delta m}(t,Y^\mu_t, Z^\mu_t, \mu_t )(a)\Big) \epsilon_n(\pi_t(a)-\mu_t(a)) da d \lambda dt\Big] \Big|^2 \notag \\
            & \leq \Big|\mathbb{E} \Big[ \int_0^T \int_0^1 \int K(|Y^{\epsilon_n}_t - Y^\mu_t|+ |Z^{\epsilon_n}_t-Z^\mu_t| ) \epsilon_n(\pi_t(a)+\mu_t(a)) da d \lambda dt\Big] \Big|^2 \notag \\
            & \leq C_K \mathbb{E}\left[\int_0^T (|Y^{\epsilon_n}_t - Y^\mu_t|^2+ |Z^{\epsilon_n}_t-Z^\mu_t|^2) dt\right]\epsilon_n^2= o(\epsilon_n^2) , \notag
        \end{align}}
        which together with (\ref{eq: estimate f2}) leads to
            \begin{equation*}\label{eq: estimate f4}
            \begin{aligned}
            &\mathbb{E} \Big[ \int_0^T \int_0^1 \int \left(\frac{\delta f}{\delta m}(t,Y^{\epsilon_n}_t, Z^{\epsilon_n}_t, (1-\lambda)\mu_t + \lambda \mu^{\epsilon_n}_t)(a)-  \frac{\delta f}{\delta m}(t,Y^\mu_t, Z^\mu_t, \mu_t )(a)\right)\\
            &\quad \quad \quad \quad \quad \ \ \cdot \epsilon_n(\pi_t(a)-\mu_t(a)) da d \lambda dt\Big] \\
            & = o(\epsilon_n).
            \end{aligned}
        \end{equation*}

        As a result,
        \begin{equation}\label{eq: limit}
            \lim_{n \rightarrow \infty} |h_1(t,\epsilon_n) + h_2(t,\epsilon_n)| =0\ \ {\rm a.e.} \ {\rm a.s.}
        \end{equation}
        By (\ref{eq: dominate}), (\ref{eq: dominate integrable}), (\ref{eq: limit}) and the dominated convergence theorem, we have
        \begin{equation*}
            \lim_{n \rightarrow \infty} \mathbb{E} \left[ \int_0^T \left(h_1(t,\epsilon_n) + h_2(t,\epsilon_n) \right) dt\right] =0.
        \end{equation*}
Hence
     $
            \lim_{n \rightarrow \infty} h(\epsilon_n) =0$ follows from (\ref{eq:h(epsilon)}).

   Due to the arbitrariness of $\epsilon_n$ and Heine Theorem, it yields that
       $
            \lim_{\epsilon \rightarrow 0^{+}} h(\epsilon) =0$.
        \end{proof}

The following lemma comes from Lemma 3.2 in \cite{siska-2020} and will be used in the variation inequality of cost functional.
    \begin{lemma}\label{lem: semi continuity of entropy 01}
For $\pi, \mu,\gamma \in \mathcal{A}$, set $\mu^{\varepsilon}=\mu+\varepsilon(\pi-\mu)$. Then\\
i) for any $\varepsilon \in(0,1)$,
{\small$$
\frac{1}{\varepsilon} \int_0^T\left(Ent\left(\mu_t^e \mid \gamma_t\right)-Ent\left(\mu_t \mid \gamma_t\right)\right) d t \geq \int_0^T \int\left(\ln \mu_t(a)-\ln \gamma_t(a)\right)\left(\pi_t(a)-\mu_t(a)\right)d a d t;
$$}
ii)
\begin{align*}
&\limsup _{\varepsilon \rightarrow 0} \frac{1}{\varepsilon} \int_0^T\left(Ent\left(\mu_t^\epsilon \mid \gamma_t\right)-Ent\left(\mu_t \mid \gamma_t\right)\right) d t\\
&\leq \int_0^T \int\left(\ln \mu_t(a)-\ln \gamma_t(a)\right)\left(\pi_t(a)-\mu_t(a)\right)d a d t.
\end{align*}
\end{lemma}

Before we prove the variation inequality of cost functional, we give an expression for the G\^{a}teaux derivative of $J^0$ in terms of the functional derivative of the Hamiltonian.
\begin{proposition}\label{proposition: variation ineq1}
        Under Assumption \ref{H: bsde},
        \begin{equation*}
        \begin{aligned}
             \lim_{\epsilon \rightarrow 0}\frac{J^0(\mu^{\epsilon}) -J^0(\mu)}{\epsilon} &= \mathbb{E}\Big[ \nabla_y \phi (Y_0)V_0 +  \int_0^T ( \nabla_y l(t)V_t + \nabla_z l(t)Z^V_t \\
             &  \quad \quad + \int\frac{\delta l}{\delta m}(t, Y^\mu_t, Z^\mu_t, \mu_t)(a)(\pi_t(a) - \mu_t(a))da )dt \Big].
        \end{aligned}
        \end{equation*}
    \end{proposition}
\begin{proof}By the variation equation (\ref{eq:variation eq}), we have
        {\small\begin{align}\label{eq:diff func}
            &J^0(\mu^{\epsilon}) -J^0(\mu) \notag\\
            & = \mathbb{E}\Big[\int_0^T\Big(l(t,Y^{\epsilon}_t, Z^{\epsilon}_t, \mu_t^{\epsilon})-l(t,Y^\mu_t, Z^\mu_t, \mu_t)\Big)dt + \phi(Y^{\epsilon}_0) - \phi(Y^\mu_0)\Big] \notag\\
            & = \mathbb{E}\Big[\int_0^T \Big(\int_0^1 \int \frac{\delta l}{\delta m}(t,Y^{\epsilon}_t, Z^{\epsilon}_t, (1-\lambda)\mu_t + \lambda \mu^{\epsilon}_t)(a) \epsilon\left(\pi_t(a)-\mu_t(a)\right) da d \lambda \notag\\
            & \quad \quad \quad \quad \quad+ \nabla_z \tilde{l}^{\epsilon}(t)(Z^{V^{\epsilon}}_t + \epsilon Z^{V}_t)+  \nabla_y \tilde{l}^{\epsilon}(t)(V^{\epsilon}_t + \epsilon V_t)\Big)dt \notag\\
            & \ \quad \quad \ + \int_0^1 \nabla_y \phi\left(Y^\mu_0 + \lambda \epsilon (V^{\epsilon}_0 + V_0)\right)(V^{\epsilon}_0 + \epsilon V_0)d\lambda\Big] \notag\\
             & = \epsilon \mathbb{E}\Big[\int_0^T \Big( \nabla_y l(t)V_t + \nabla_z l^{}(t)Z^V_t+ \int\frac{\delta l}{\delta m}(t, Y^\mu_t, Z^\mu_t, \mu_t)(a)(\pi_t(a) - \mu_t(a))da \Big)dt \notag\\
             & \quad \quad \quad + \nabla_y \phi (Y^\mu_0)V_0\Big]+ \mathbb{E}\Big[\int_0^T \big(\nabla_y l(t)V^{\epsilon}_t  + \nabla_z l(t)Z^{V^{\epsilon}}_t )dt\Big]\notag\\
             & \quad + \mathbb{E}\Big[\int_0^T \Big((\nabla_y \tilde{l}^{\epsilon}(t)-\nabla_y l(t))(V^{\epsilon}_t + \epsilon V_t)  + (\nabla_z \tilde{l}^{\epsilon}(t)-\nabla_z l(t))(Z^{V^{\epsilon}}_t + \epsilon Z^{V}_t) \notag\\
            & \quad + \mathbb{E} \Big[ \int_0^T \int_0^1 \int \Big(\frac{\delta l}{\delta m}(t,Y^{\epsilon}_t, Z^{\epsilon}_t,(1-\lambda)\mu_t + \lambda \mu^{\epsilon}_t)(a)-  \frac{\delta l}{\delta m}(t,Y^{\epsilon}_t, Z^{\epsilon}_t, \mu_t )(a)\Big)\notag\\
            & \quad \quad \quad \quad \quad \quad \quad \quad \ \ \ \cdot \epsilon \left(\pi_t(a)-\mu_t(a) \right) da d \lambda dt\Big] \notag\\
            & \quad + \mathbb{E} \Big[ \int_0^T \int_0^1 \int \Big(\frac{\delta l}{\delta m}(t,Y^{\epsilon}_t, Z^{\epsilon}_t, \mu_t)(a) -  \frac{\delta l}{\delta m}(t,Y^\mu_t, Z^\mu_t, \mu_t )(a)\Big)\epsilon \left(\pi_t(a)-\mu_t(a) \right) da d \lambda dt\Big] \notag\\
            & \quad + \mathbb{E}\Big[ \nabla_y \phi(Y^\mu_0)V^{\epsilon}_0 +\int_0^1  \left(\nabla_y \phi^{}\left(Y^\mu_0 + \lambda  (V^{\epsilon}_0 + \epsilon V_0)\right) - \nabla_y \phi^{}(Y^\mu_0)\right)(V^{\epsilon}_0 + \epsilon V_0)d \lambda \Big]. \notag\\
        \end{align}}
First by Assumption \ref{H: bsde} and Lemma \ref{lem: small order estimate }, we know
{\small\begin{equation*}
            \begin{aligned}
                \Big| \mathbb{E} \Big[\nabla_y \phi(Y^\mu_0)V^{\epsilon}_0\Big]\Big|^2 &\leq \mathbb{E} \Big[|\nabla_y \phi(Y^\mu_0)|^2\Big] \cdot \mathbb{E} \Big[|V^{\epsilon}_0|^2\Big]\leq 2K^2\mathbb{E} \Big[1+ |Y^\mu_0|^2\Big] \mathbb{E} \Big[\sup_{0\leq t \leq T}|V^{\epsilon}_t|^2\Big]\\
                &\leq 2K^2 \Big(1+ \mathbb{E}\Big[\sup_{0\leq t \leq T}|Y^\mu_t|^2\Big]\Big) \mathbb{E} \Big[\sup_{0\leq t \leq T}|V^{\epsilon}_t|^2\Big]= o(\epsilon^2)
            \end{aligned}
        \end{equation*}}
and
        {\small\begin{equation*}
        \begin{aligned}
        \left|\mathbb{E} \left[\int_0^T \nabla_y l (t)V_t^{\epsilon}d t\right]\right|^2
        & \leq \mathbb{E} \left[\int_0^T\left|\nabla_y l(t)\right|^2 d t \right]\cdot \mathbb{E}\left[ \int_0^T\left|V_t^{\epsilon}\right|^2 d t\right] \\
        & \leq 3 K^2 \mathbb{E} \left[\int_0^T\left(1+\left|Y^\mu_t\right|^2+\left|Z^\mu_t\right|^2\right) d t \right]\cdot \mathbb{E} \left[\int_0^T\left|V_t^{\epsilon}\right|^2 d t\right] \\
        & \leq C_{\mu, T, K} \mathbb{E} \left[\int_0^T\left|V_t^{\epsilon}\right|^2 d t\right]=o\left(\varepsilon^2\right) .
        \end{aligned}
        \end{equation*}}
Similarly, $\left|\mathbb{E} \Big[\int_0^T \nabla_z l(t)Z^{V^{\epsilon}}_t d t\Big]\right|^2=o\left(\varepsilon^2\right)$.
                   
        Hence
        \begin{equation*}
            \left\{\begin{array}{l}
        \mathbb{E} [\nabla_y \phi(Y^\mu_0)V^{\epsilon}_0]=o(\varepsilon), \\
        \mathbb{E} [\int_0^T \nabla_y l (t)V_t^{\epsilon}d t]=o(\varepsilon), \\
        \mathbb{E} [\int_0^T \nabla_z l(t)Z^{V^{\epsilon}}_t d t]=o(\varepsilon).
        %,\\
        %\mathbb{E} [\int_0^T \nabla_{\bar{z}} l(t)\bar{Z}^{V^{\epsilon}}_t d t] =o(\varepsilon).
        \end{array}\right.
        \end{equation*}

        On the other hand, by Lemma \ref{lem:big order estimate} we have
        \begin{equation}\label{eq:diff01}
            \begin{aligned}
                & \Big| \mathbb{E}\Big[\int_0^T \left(\nabla_y \tilde{l}^{\epsilon}(t)-\nabla_y l(t)\right)(V^{\epsilon}_t + \epsilon V_t)dt\Big]\Big|^2 \\
                &  = \Big|\mathbb{E}\Big[\int_0^T \left(\nabla_y \tilde{l}^{\epsilon}(t)-\nabla_y l(t)\right)(Y_t^{\epsilon} - Y^\mu_t)dt\Big]\Big|^2\\
                &  \leq \mathbb{E}\Big[\int_0^T |\nabla_y \tilde{l}^{\epsilon}(t)-\nabla_y l(t)|^2dt\Big]  \cdot \mathbb{E}\Big[\int_0^T |Y_t^{\epsilon} - Y^\mu_t|^2 dt\Big]\\
                &  \leq C  \epsilon^2 \cdot \mathbb{E}\Big[\int_0^T |\nabla_y \tilde{l}^{\epsilon}(t)-\nabla_y l(t)|^2dt\Big] .
            \end{aligned}
        \end{equation}
Then we prove $\lim\limits_{\epsilon \rightarrow 0} \mathbb{E}\left[\int_0^T |\nabla_y \tilde{l}^{\epsilon}(t)-\nabla_y l(t)|^2dt\right] = 0$. For any $0 \leq \lambda \leq 1$, Set
\begin{equation*}\label{eq: dm01}
        \begin{aligned}
                  r_y(t, \epsilon; \lambda):= |\nabla_y l(t, Y^\mu_t + \lambda (Y^{\epsilon}_t - Y^\mu_t), Z^\mu_t + \lambda (Z^{\epsilon}_t - Z^\mu_t),  \mu_t) - \nabla_y l(t)|^2.
        \end{aligned}
        \end{equation*}
        Similar to the proof of Lemma \ref{lem: small order estimate }, we get from Lemma \ref{lem:big order estimate} that  
        there exists a subsequence of  $\{\epsilon_n\}$, still denoted by $\{\epsilon_n\}$, such that $(Y^{\epsilon_n}_t, Z^{\epsilon_n}_t) $ converges to $(Y^\mu_t, Z^\mu_t) $ a.e. a.s. By the continuity of $\nabla_y l$ with respect to $(y,z)$, we have for any $0 \leq \lambda \leq 1$,
        \begin{equation*}\label{eq:int01}
            \lim_{n \rightarrow \infty} r_y(t, \epsilon_n; \lambda) = 0\ \ {\rm a.e.} \ {\rm a.s.}
        \end{equation*}
Moreover, as $n$ large enough,
        \begin{equation*}
        \begin{aligned}
                 &\mathbb{E}\Big[\int_0^T r_y(t, \epsilon_n; \lambda)\Big]\\&\leq 24K^2\mathbb{E}\Big[\int_0^T\Big((1+|Y^\mu_t|^2 +|Z^\mu_t|^2 ) + 16K^2(|Y^{\epsilon_n}_t-Y^\mu_t|^2+ |Z^{\epsilon_n}_t-Z^\mu_t|^2 )\Big)dt\Big]\\
                  &\leq C\mathbb{E}\Big[\int_0^T(1+|Y^\mu_t|^2 +|Z^\mu_t|^2 )dt\Big]<\infty.
                 %:= g_y(t, \epsilon)\quad \ {a.e.} \ {a.s.}
        \end{aligned}
        \end{equation*}
        
Notice
        {\small\begin{equation*}\label{eq:dm11}
            \begin{aligned}
            &\mathbb{E}\Big[\int_0^T|\nabla_y \tilde{l}^{\epsilon_n}(t)-\nabla_y l(t)|^2 dt\Big] \\
            &= \mathbb{E}\Big[\int_0^T\Big|\int_0^1 \bigg(\nabla_y l(t, Y^\mu_t + \lambda (Y^{\epsilon_n}_t - Y^\mu_t), Z^\mu_t + \lambda (Z^{\epsilon_n}_t - Z^\mu_t), \mu_t)- \nabla_y l(t) \bigg) d \lambda \Big|^2 dt\Big]\\
            &\leq \mathbb{E}\Big[\int_0^T\int_0^1 r_y(t, \epsilon_n; \lambda) d \lambda\Big].%\\
            %&\leq g_y(t, \epsilon_n), \quad \forall n\geq 1, \quad {\rm a.e.\ a.s}
        \end{aligned}
        \end{equation*}}
By the dominated convergence theorem again, we have
        
        \begin{equation*}
            \lim_{n \rightarrow \infty} \mathbb{E}\Big[\int_0^T |\nabla_y \tilde{l}^{\epsilon_n}(t)-\nabla_y l(t)|^2dt\Big]  = 0.
        \end{equation*}
The arbitrariness of $\left\{\varepsilon_n\right\}$ and Heine theorem leads to
        \begin{equation*}
            \lim_{\epsilon \rightarrow 0} \mathbb{E}\Big[\int_0^T |\nabla_y \tilde{l}^{\epsilon}(t)-\nabla_y l(t)|^2dt\Big]  = 0.
        \end{equation*}
Back to (\ref{eq:diff01}), the above deductions imply
        \begin{equation*}
        \begin{aligned}
            \mathbb{E}\Big[\int_0^T \left(\nabla_y \tilde{l}^{\epsilon}(t)-\nabla_y l(t) \right)(V^{\epsilon}_t + \epsilon V_t)dt\Big] = o(\epsilon).
        \end{aligned}
        \end{equation*}

By similar deductions, we also have
        \begin{equation}\label{eq: estimate1}
        \left\{\begin{aligned}
            %&\mathbb{E}[\int_0^T (\nabla_y \tilde{l}^{\epsilon}(t)-\nabla_y l(t))(V^{\epsilon}_t + \epsilon V_t)dt] = o(\epsilon),\\
            &\mathbb{E}\Big[\int_0^T \left(\nabla_z \tilde{l}^{\epsilon}(t)-\nabla_z l(t)\right)(Z^{V^{\epsilon}}_t + \epsilon Z^V_t)dt\Big] = o(\epsilon),\\
            %&\mathbb{E}\Big[\int_0^T (\nabla_{\bar{z}} \tilde{l}^{\epsilon}(t)-\nabla_{\bar{z}} l(t))(\bar{Z}^{V^{\epsilon}}_t + \epsilon \bar{Z}^V_t)dt\Big] = %o(\epsilon),\\
            &\mathbb{E} \Big[ \int_0^1  \left(\nabla_y \phi(Y^\mu_0 + \lambda \epsilon (V^{\epsilon}_0 + V_0)) - \nabla_y \phi(Y^\mu_0)\right)(V^{\epsilon}_0 + \epsilon V_0)d \lambda\Big] = o(\epsilon) .
        \end{aligned}\right.
        \end{equation}

Then we deal with two terms involving $\frac{\delta l}{\delta m}$. For the first one, it turns out that 
        \begin{equation}\label{eq: estimate2}
            \begin{aligned}
            &\Big|\mathbb{E} \Big[ \int_0^T \int_0^1 \int \Big(\frac{\delta l}{\delta m}(t,Y^{\epsilon}_t, Z^{\epsilon}_t, (1-\lambda)\mu_t + \lambda \mu^{\epsilon}_t)(a) \\
            & \quad \quad \quad \quad \quad \quad \ \ -  \frac{\delta l}{\delta m}(t,Y^{\epsilon}_t, Z^{\epsilon}_t,  \mu_t )(a)\Big)\epsilon(\pi_t(a)-\mu_t(a)) da d \lambda dt\Big] \Big|^2\\
            & \leq \mathbb{E} \Big[ \int_0^T \Big| \int_0^1 \int_0^1 \lambda \int \int \frac{\delta^2 l}{\delta m^2}(t,Y^{\epsilon}_t, Z^{\epsilon}_t, \mu_t^{\lambda,\lambda^{\prime}})(a,a^{\prime})  \epsilon(\pi_t(a^{\prime})-\mu_t(a^{\prime})) d a^{\prime}\\
            &\quad \quad \quad \quad \quad \quad \quad \quad \quad \quad \quad \cdot \epsilon(\pi_t(a)-\mu_t(a)) da d \lambda d \lambda^{\prime}\Big|^2 dt\Big]\\
            &\leq \mathbb{E}  \Big[\int_0^T \Big| \int_0^1 \int_0^1 \lambda  K\int (\pi_t(a^{\prime})+\mu_t(a^{\prime})) d a^{\prime} \int(\pi_t(a)+\mu_t(a)) da d \lambda d \lambda^{\prime}\Big|^2 dt \Big] \epsilon^2\\
            & = C_{K,T}\epsilon^2 = o(\epsilon).
            \end{aligned}
        \end{equation}
For the other term, since $\frac{\delta l}{\delta m}$ is uniformly Lipschitz in $(y,z)$, % for any fixed $m \in \mathcal{A}$,
by Lemma \ref{lem:big order estimate} we have
       \begin{equation}\label{eq: estimate3}
            \begin{aligned}
            &\Big|\mathbb{E} \Big[ \int_0^T \int_0^1 \int (\frac{\delta l}{\delta m}(t,Y^{\epsilon}_t, Z^{\epsilon}_t, \mu_t)(a) -  \frac{\delta l}{\delta m}(t,Y^\mu_t, Z^\mu_t, \mu_t )(a))\epsilon(\pi_t(a)-\mu_t(a)) da d \lambda dt\Big] \Big|^2\\
            &\leq \Big|\mathbb{E} \Big[ \int_0^T \int_0^1 \int K(|Y^{\epsilon}_t - Y^\mu_t|+ |Z^{\epsilon}_t-Z^\mu_t| )\epsilon(\pi_t(a)+\mu_t(a)) da d \lambda dt\Big] \Big|^2\\
            &\leq C_K \mathbb{E}\Big[\int_0^T (|Y^{\epsilon}_t - Y^\mu_t|^2+ |Z^{\epsilon}_t-Z^\mu_t|^2 ) dt\Big]\epsilon^2  = o(\epsilon).
            \end{aligned}
        \end{equation}

Therefore, Proposition \ref{proposition: variation ineq1} follows from  (\ref{eq:diff func}) and (\ref{eq: estimate1})--(\ref{eq: estimate3}).
    \end{proof}

Now we are ready to present and prove the variation equality of cost functional.
  \begin{proposition}\label{proposition: variation ineq2}
        Under Assumption \ref{H: bsde},
       {\small \begin{equation*}
        \begin{aligned}
            \mathbb{E}\Big[ \nabla_y \phi (Y_0)V_0 +  \int_0^T \Big( & \nabla_y l^{}(t)V_t + \nabla_z l^{}(t)Z^V_t+ \int\frac{\delta l}{\delta m}(t, Y^\mu_t, Z^\mu_t, \mu_t)(a)(\pi_t(a) - \mu_t(a))da  \\
             & + \frac{\sigma^2}{2} \int \left((\ln \mu_t(a) + U(a))(\pi_t(a) - \mu_t(a))\right)da\Big)dt\Big] \geq 0.\\
        \end{aligned}
        \end{equation*} }
        where $\mu$ is the optimal control for (\textbf{P1}).
    \end{proposition}

    \begin{proof}
        From the optimality of $\mu$, based on Lemma \ref{lem: semi continuity of entropy 01} and Proposition \ref{proposition: variation ineq1} we have
        \begin{align*}\label{eq:J sigma difference for variation ineq}
             0 & \leq \limsup_{\epsilon \rightarrow 0^+}\frac{J^\sigma(\mu^{\epsilon}) -J^\sigma(\mu)}{\epsilon} \notag \\
             &= \limsup_{\epsilon \rightarrow 0^+}\left(\frac{J^0(\mu^{\epsilon}) -J^0(\mu)}{\epsilon} + \frac{1}{\varepsilon} \mathbb{E}\left[\int_0^T\left(Ent\left(\mu_t^e \mid U\right)-Ent\left(\mu_t \mid U\right)\right) d t\right]\right) \notag \\
             & \leq \mathbb{E}\Big[ \nabla_y \phi (Y_0)V_0 +  \int_0^T \Big( \nabla_y l(t)V_t + \nabla_z l(t)Z^V_t \notag \\
             &  \quad \quad\quad \quad\quad \quad\quad \quad\quad \quad\quad \ \ + \int\frac{\delta l}{\delta m}(t, Y^\mu_t, Z^\mu_t, \mu_t)(a)(\pi_t(a) - \mu_t(a))da  \notag\\
             &  \quad \quad\quad \quad\quad \quad\quad \quad\quad \quad\quad \ \  + \frac{\sigma^2}{2} \int \left((\ln \mu_t(a) + U(a))(\pi_t(a) - \mu_t(a))\right)da\Big)dt\Big].
        \end{align*}
    \end{proof}

    %\subsubsection{Maximum Principle}

    Now we use duality technique to derive the local necessary condition for optimal control. To begin with, we introduce the adjoint equation of variation equation (\ref{eq:variation eq})
    \begin{equation}\label{eq:duality equation}
        \left\{\begin{aligned}
            &d P^\mu_t = -\left(-\nabla_y f^{\top}(t)P^\mu_t + \nabla_y l(t)\right)dt - \left(-\nabla_z f^{\top}(t)P^\mu_t + \nabla_z l(t)\right)d W_t \\
            %& \quad \quad \quad - \left(-\nabla_{\bar{z}} f(t)P^\mu_t + \nabla_{\bar{z}} l(t)\right)d\bar{W}_t,\\
            & P^\mu_0 = -\nabla_y \phi(Y^\mu_0).
        \end{aligned}\right.
    \end{equation}
Actually, SDE (\ref{eq:duality equation}) is a linear equation whose coefficients $\nabla_y f^{\top}(t)$ and $\nabla_z f^{\top}(t)$ are duality operators of $\nabla_y f(t)$ and $\nabla_z f(t)$ satisfying Lipschitz conditions, so it is clear that SDE (\ref{eq:duality equation}) has a unique solution in $S^2_{\mathscr{F}}(0,T;\mathbb{R}^{n})$. Then introduce the Hamiltonians $H^0: \Omega \times [0,T] \times \mathbb{R}^n \times \mathbb{R}^{n \times m} \times \mathbb{R}^n \times \mathcal{P}(\mathbb{R}^p) \rightarrow \mathbb{R}$ and $H^{\sigma}: \Omega \times [0,T] \times \mathbb{R}^n \times \mathbb{R}^{n \times m} \times \mathbb{R}^n \times \mathcal{P}(\mathbb{R}^p) \times \mathcal{P}(\mathbb{R}^p) \rightarrow \mathbb{R}$ as follows
    \begin{equation*}\label{eq:hamiltonian}
        \begin{array}{l}
        H^0(t,y,z,p,m):= -pf(t,y,z,m) + l(t,y,z,m), \\
        H^{\sigma}(t,y,z,p,m,m^{\prime}):= H^0(t,y,z,p,m) + \frac{\sigma^2}{2}Ent(m \mid m^{\prime}).
        \end{array}
    \end{equation*}
It can be seen from Assumption \ref{H: bsde} that
    $H^0(t,y,z,p,m)$ and $H^{\sigma}(t,y,z,p,m,m^{\prime})$ are continuous with respect to $(y,z,p)$ and differentiable with respect to $(y,z)$. By Definition \ref{def: flat derivative} $H^0(t,y,z,p,m)$ has flat derivative
    \begin{equation*}\label{eq:flat deriviative of H0}
    \frac{\delta H^{0}(t,y,z,p,m) }{\delta m}:= \frac{\delta H^{0}}{\delta m}(t,y,z,p,m)(a).
    \end{equation*}
   Hence the adjoint equation (\ref{eq:duality equation}) is equivalent to
        \begin{equation*}\label{eq:duality equation2}
        \left\{\begin{aligned}
            &d P^\mu_t = -\nabla_y H^{0} (t, Y^\mu_t, Z^\mu_t, P^\mu_t, \mu_t)dt -\nabla_z H^{0} (t, Y^\mu_t, Z^\mu_t, P^\mu_t, \mu_t)d W_t \\
            %& \quad \quad \quad -\nabla_{\bar{z}} H^{0} (t, Y^\mu_t, Z^\mu_t, \bar{Z}^\mu_t, P^\mu_t, \mu_t)d\bar{W}_t,\\
            & P_0 = -\nabla_y \phi(Y_0).
        \end{aligned}\right.
    \end{equation*}
    Although the term of entropy is lower-semi continuous and does not have flat derivative, we still use the following notation for convenience
    \begin{equation*}\label{eq:flat deriviative of Hsigma}
    \frac{\delta H^{\sigma}}{\delta m}(t,y,z,p,m)(a) := \frac{\delta H^{0}}{\delta m}(t,y,z,p,m)(a) + \frac{\sigma^2}{2}(U(a) + \ln m(a) + 1).
    \end{equation*}

   Based on (\ref{eq:variation eq}) and adjoint equation (\ref{eq:duality equation}), the variation inequality of cost functional in Proposition \ref{proposition: variation ineq1} can be written in a new form.
     \begin{proposition}\label{proposition: variation ineq3}
        Under Assumption \ref{H: bsde},
        \begin{align*}
             \lim_{\epsilon \rightarrow 0^+}\frac{J^0(\mu^{\epsilon}) -J^0(\mu)}{\epsilon}= \mathbb{E}\Big[ \int_0^T \int \Big(\frac{\delta H^{0}}{\delta m}(t,Y^\mu_t, Z^\mu_t, P^\mu_t,\mu_t)(a) \Big) (\pi_t(a) - \mu_t(a)) da dt\Big].
        \end{align*}
    \end{proposition}
    \begin{proof}
        Applying It\^{o} formula to $P^\mu_tV_t$, where %$P_t$ is the solution to adjoint equation (\ref{eq:duality equation}) and
        $V$ is the solution to variation equation (\ref{eq:variation eq}), we have
            {\small\begin{align}\label{eq: PV}
                & d P_t^{\mu}V_t \notag \\
                & = V_tdP^\mu_t + P^\mu_td V_t + d P^\mu_td V_t \notag \\
                & = \Bigg(-V_t\left(-\nabla_y f(t)P^\mu_t + \nabla_y l(t) \right)- P^\mu_t \left(\nabla_y f(t)  V_t + \nabla_z f(t) Z^{V}_t\right) \notag \\
                & \quad \ \ \ \ - Z_t^V \left(-\nabla_z f(t)P^\mu_t + \nabla_z l(t) \right)- P^\mu_t\int \frac{\delta f}{\delta m}(t, Y^\mu_t, Z^\mu_t,  \mu_t)(a)(\pi_t(a) - \mu_t(a))da\Bigg)dt \notag \\
                &  \quad \ \ + \left(P_t^{\mu}Z^{V}_t -V_t(-\nabla_z f(t)P^\mu_t + \nabla_z l(t))\right) d W_t. % \notag \\
                %& \quad \ \ + \left(P_t^{\mu}\bar{Z}^{V}_t -V_t(-\nabla_{\bar{z}} f(t)P^\mu_t + \nabla_{\bar{z}} l(t))\right) d \bar{W}_t.
            \end{align}}
Taking expectation on both sides of (\ref{eq: PV}) and using standard stopping time arguments, we obtain
    \begin{equation}\label{eq:expectation}
        \begin{aligned}
&\mathbb{E}\Big[\nabla_y\phi(Y^\mu_0)V_0 + \int_0^T\Big(\nabla_y l(t)V_t + \nabla_z l(t)Z^V_t\\
            & \quad \quad \quad \quad \quad \quad \quad \quad \quad \quad \ + P^\mu_t\int \frac{\delta f}{\delta m}(t, Y^\mu_t, Z^\mu_t, \mu_t)(a)(\pi_t(a) - \mu_t(a))da\Big)dt\Big]=0.
        \end{aligned}
    \end{equation}
Then, based on Proposition \ref{proposition: variation ineq1} and (\ref{eq:expectation}), we have
        \begin{align*}\label{cz1}
             &\lim_{\epsilon \rightarrow 0^+}\frac{J^0(\mu^{\epsilon}) -J^0(\mu)}{\epsilon}\\
             &= \mathbb{E}\Big[ \int_0^T \Big(\nabla_y l(t)V_t + \nabla_z l(t)Z^V_t\\
             & \quad \quad \quad \quad \quad+ \int\frac{\delta l}{\delta m}(t, Y^\mu_t, Z^\mu_t, \mu_t)(a)(\pi_t(a) - \mu_t(a))da \Big)dt + \nabla_y \phi (Y^\mu_0)V_0 \Big] \\
             & = \mathbb{E}\Big[\int_0^T \Big( -P_t^{\mu}\int \frac{\delta f}{\delta m}(t, Y^\mu_t, Z^\mu_t, \mu_t)(a)(\pi_t(a) - \mu_t(a))da \\
             & \quad \quad \quad \quad \quad + \int\frac{\delta l}{\delta m}(t, Y^\mu_t, Z^\mu_t, \mu_t)(a)(\pi_t(a) - \mu_t(a))da \Big)dt\Big]\\
             & = \mathbb{E}\Big[\int_0^T\int \frac{\delta H^{0}}{\delta m}(t,Y^\mu_t,Z^\mu_t, P^\mu_t,\mu_t)(a) \cdot (\pi_t(a) - \mu_t(a)) da dt\Big].
        \end{align*}
    \end{proof}

Now we are ready to prove the maximum principle for Problem (\textbf{P1}).
    \begin{theorem}\label{cz4}
        Under Assumption \ref{H: bsde}, if $\mu \in \mathcal{A}$ is an optimal control of Problem (\textbf{P1}) with the corresponding optimal state process $(Y^\mu, Z^\mu) $, and $P^\mu$ is the solution to adjoint equation (\ref{eq:duality equation}), then for any $t\in[0,T]$, $\pi \in \mathcal{A}$,
        \begin{equation}\label{eq: general smp}
           \int \Big(\frac{\delta H^{0}}{\delta m}(t,Y^\mu_t, Z^\mu_t, P^\mu_t,\mu_t)(a) + \frac{\sigma^2}{2}(\ln \mu_t(a) +U(a)) \Big) (\pi_t(a) - \mu_t(a)) da \geq 0  \quad {\rm a.s.}
        \end{equation}
    \end{theorem}
    \begin{proof}
     Based on Propositions \ref{proposition: variation ineq2} and \ref{proposition: variation ineq3}, we deduce from the optimality of $\mu$ that
        {\small\begin{equation*}
            \begin{aligned}
            &\mathbb{E}\Big[ \int_0^T \int\Big(\frac{\delta H^0}{\delta m}\left(t,Y^\mu_t,Z^\mu_t,P^\mu_t,\mu_t\right)(a)+\frac{\sigma^2}{2} \left(\ln \mu_t(a)+  U(a)\right)\Big) \left(\pi_t(a)-\mu_t(a)\right)d a d t \Big] \geq 0.
            \end{aligned}
        \end{equation*}}

Assume that \eqref{eq: general smp} doesn't hold. This means that there is a $\tilde{\pi} \in \mathcal{A}$ and $S_{\epsilon} \in \mathscr{F} \otimes \mathcal{B}([0, T])$ with a strictly positive measure $\mathbb{P} \otimes \Lambda$, where $\Lambda$ is the Lebesgue measure on $\mathcal{B}([0, T])$ and
         \begin{align*}
         S_{\epsilon} = \bigg\{(\omega, t): \int &\Big(\frac{\delta H^{0}}{\delta m}(t,Y^\mu_t,Z^\mu_t,P^\mu_t,\mu_t)(a) + \frac{\sigma^2}{2}(\ln \mu_t(a) + U(a)) \Big)\\
         &\cdot(\tilde{\pi}_t(a) - \mu_t(a)) da \leq -\epsilon < 0\bigg\}.
         \end{align*}
        Define $\tilde{\mu}_t:=\tilde{\pi}_t \mathbb{I}_{S_{\epsilon}}+\mu_t \mathbb{I}_{{S_{\epsilon}^c}}$. We have
        {\small$$
        \begin{aligned}
        0 & \leq \mathbb{E} \Big[\int_0^T \int\Big(\frac{\delta H^0}{\delta m}\left(t,Y^\mu_t,Z^\mu_t,P^\mu_t,\mu_t\right)(a)+\frac{\sigma^2}{2}\left(\ln \mu_t(a)+ U(a)\right)\Big)\left(\tilde{\mu}_t(a)-\mu_t(a)\right)d a d t \Big] \\
        & =\mathbb{E} \Big[\int_0^T \mathbb{I}_{S_{\epsilon}} \int\Big(\frac{\delta H^0}{\delta m}\left(t,Y^\mu_t,Z^\mu_t,P^\mu_t,\mu_t\right)(a)+\frac{\sigma^2}{2}\left(\ln \mu_t(a)+U(a)\right)\Big)\\
        &\ \ \ \ \ \ \ \ \ \ \ \ \ \ \ \ \ \ \ \ \ \ \cdot\left(\tilde{\pi}_t(a)-\mu_t(a)\right)d a d t \Big] \\
        & \leq-\epsilon \mathbb{E}\int_0^T \mathbb{I}_{S_{\epsilon}}dt < 0,
        \end{aligned}
        $$}
        which leads to a contradiction. Then the proof follows.
    \end{proof}

\section{Further Discussions of Optimal Controls}
    %\subsubsection{Sufficient Condition for Optimal Control}

In this section we present a sufficient condition for the optimal control and give an implicit form of it. For the sufficient condition, the convex conditions for the coefficients are needed.
    \begin{ass}\label{ass: convex}
    For the coefficients in Problem (\textbf{P1}), $\phi(y)$ is convex in $y$ and $H^{\sigma}(t,y,z,p,m,m')$ 
    is convex in $(y,z,m)$. 
    \end{ass}

    With Assumption \ref{ass: convex} we prove the sufficient condition for the optimal control.
    \begin{theorem}\label{thm:verification theorem1}
         Under Assumptions \ref{H: bsde} and \ref{ass: convex}, %for a given $\sigma \geq 0$,
         the control $\mu \in \mathcal{A}$ is an optimal control of Problem (\textbf{P1}) if for any $t\in[0,T]$, $\pi \in \mathcal{A}$, it satisfies \eqref{eq: general smp} with $(Y^\mu, Z^\mu)$ and $P^\mu$ be the solutions to the corresponding BSDE (\ref{equation:bsde1}) and adjoint equation (\ref{eq:duality equation}), respectively.
    \end{theorem}
    \begin{proof}
      For $\mu\in\mathcal{A}$ satisfying (\ref{eq: general smp}), we have
      \begin{equation}\label{cz2}
            \begin{aligned}
                J^{\sigma}(\pi) - J^{\sigma}(\mu)= I_1 + I_2,
            \end{aligned}
        \end{equation}
 where  $
                I_1  = \mathbb{E}\Big[\phi(Y_0^{\pi}) - \phi(Y_0^{\mu})\Big]
        $
and
        {\small\begin{equation}\label{eq:i2}
            \begin{aligned}
                I_2  =  \mathbb{E}\Big[\int_0^T [&l(t, Y_t^{\pi}, Z_t^{\pi}, \pi_t) + \frac{\sigma^2}{2}Ent(\pi_t \mid e^{-U})- l(t, Y_t^{\mu}, Z_t^{\mu}, \mu_t) - \frac{\sigma^2}{2}Ent(\mu_t \mid e^{-U})]dt\Big].
            \end{aligned}
        \end{equation}}
For $I_1$, applying It\^{o} formula to $P_t^{\mu }(Y_t^{\pi} - Y_t^{\mu})$, we have
       \begin{equation*}
            \begin{aligned}
           &d P_t^{\mu}(Y_t^{\pi} - Y_t^{\mu})\\
            & = \Bigg(-(Y_t^{\pi} - Y_t^{\mu})(-\nabla_y f(t)P^{\mu}_t + \nabla_y l(t) )- (Z_t^{\pi}-Z_t^\mu)(-\nabla_z f(t)P_t^{\mu} + \nabla_z l(t))\\
                &\quad \quad - P_t^\mu (f(t, Y_t^{\mu}, Z_t^{\mu}, \mu_t) - f(t, Y_t^{\pi}, Z_t^{\pi}, \pi_t))\Bigg)dt\\
                & \quad \ \ + \Big(P_t^{\mu}(Z_t^{\pi}-Z_t^{\mu}) -(Y_t^{\pi} - Y_t^{\mu})(-\nabla_z f(t)P_t^{\mu} + \nabla_z l(t))\Big) d W_t,
       \end{aligned}
       \end{equation*}
so by the convexity of $\phi$ it yields that
        \begin{equation}\label{eq:i1}
            \begin{aligned}
                I_1  %& = \mathbb{E}\Big[\phi(Y_0^{\pi}) - \phi(Y_0^{\mu})\Big] \\
                & \geq \mathbb{E}\Big[\nabla_y \phi(Y_0^{\mu})(Y_0^{\pi} - Y_0^{\mu})\Big]\\
                & = -\mathbb{E}\Big[P_0^{\mu }(Y_0^{\pi} - Y_0^{\mu})\Big]\\
                & = -\mathbb{E}\Big[\int_0^T (Y_t^{\pi} - Y_t^{\mu})\nabla_y  H^{0} (t, Y_t^{\mu}, Z_t^{\mu}, P_t^{\mu}, \mu_t) dt\Big] \\
                & \quad - \mathbb{E}\Big[\int_0^T (Z_t^{\pi} - Z_t^{\mu})\nabla_z  H^{0} (t, Y_t^{\mu}, Z_t^{\mu}, P_t^{\mu}, \mu_t) dt\Big] \\
                %& \quad -\mathbb{E}\Big[\int_0^T (\bar{Z}_t^{\pi} - \bar{Z}_t^{\mu})\nabla_{\bar{z}}  H^{0} (t, Y_t^{\mu}, Z_t^{\mu}, \bar{Z}^{\mu}_t, P_t^{\mu}, \mu_t) dt\Big]\\
                & \quad - \mathbb{E}\Big[\int_0^T P_t^{\mu} (f(t, Y_t^{\mu}, Z_t^{\mu}, \mu_t) - f(t, Y_t^{\pi}, Z_t^{\pi}, \pi_t))dt\Big].
            \end{aligned}
        \end{equation}
Hence by (\ref{cz2})--(\ref{eq:i1}) and the convexity of $H^{\sigma}$ we have
         \begin{align*}
                & J^{\sigma}(\pi) - J^{\sigma}(\mu) \\
                %& = I_1 + I_2\\
                & \geq \mathbb{E}\Big[\int_0^T \Big(H^{\sigma}(t, Y_t^{\pi}, Z_t^{\pi}, P_t^{\mu},\pi_t, e^{-U}) - H^{\sigma}(t, Y_t^{\mu}, Z_t^{\mu}, P_t^{\mu},\mu_t, e^{-U})\Big)dt \Big]\\
                & \quad -\mathbb{E}\Big[\int_0^T (Y_t^{\pi} - Y_t^{\mu})\nabla_y  H^{0} (t, Y_t^{\mu}, Z_t^{\mu}, P_t^{\mu}, \mu_t) dt \Big] \\
                & \quad - \mathbb{E}\Big[\int_0^T (Z_t^{\pi} - Z_t^{\mu})\nabla_z  H^{0} (t, Y_t^{\mu}, Z_t^{\mu}, P_t^{\mu}, \mu_t) dt\Big] \\
                & \geq \mathbb{E}\Big[\int_0^T \int \Big(\frac{\delta H^0}{\delta m}(t, Y_t^{\mu}, Z_t^{\mu}, P_t^{\mu}, \mu_t)(a) + \frac{\sigma^2}{2}(\ln\mu_t + U(a))\Big)(\pi_t(a) - \mu_t(a))da dt\Big].
            \end{align*}

Therefore, by the condition (\ref{eq: general smp}), it follows that $J^{\sigma}(\pi) - J^{\sigma}(\mu) \geq 0$ for any $\pi \in \mathcal{A}$, which implies that $\mu$ is the optimal control.
    \end{proof}

We then give an implicit form of optimal control of BSDE (\ref{equation:bsde1}) with the cost functional (\ref{eq:func bsde}).
Assume that an optimal control $\mu$ exists. For fixed $t\in[0,T]$ and $\omega \in \Omega$, (\ref{eq: general smp}) in Theorems \ref{cz4} and \ref{thm:verification theorem1} is equivalent to
    \begin{equation}\label{eq:optimization problem0}
        \begin{aligned}
            \mu_t \in \mathop{\arg\min}\limits_{m \in \mathcal{P}_2(\mathbb{R}^p)} H^{\sigma}(t,Y^\mu_t, Z^\mu_t, P^\mu_t,m,e^{-U}),
        \end{aligned}
         \end{equation}
         where $(Y^\mu, Z^\mu)$ and $P^\mu$ are the solutions to BSDE (\ref{equation:bsde1}) and the adjoint equation (\ref{eq:duality equation}), respectively, with the control variable $\mu$.

Assume that $U \in C^{\infty}$, $\nabla_a U$ is Lipschitz-continuous, and there exists constants $C_U >0$ and  $C^\prime_U \in \mathbb{R}$ such that for any $a \in \mathbb{R}^p$, it holds that
$
        \nabla_a U(a) \cdot a \geq C_U|a|^2 + C^\prime_U
$.       According to Proposition 2.5 in Hu, Ren, {\v S}i{\v s}ka and Szpruch \cite{hu-2019}, in which the deterministic control system is studied, the admissible control set of the optimization problem \eqref{eq:optimization problem0} can be enlarged to $\mathcal{P}(\mathbb{R}^p)$ with above assumptions on $U$. So  \eqref{eq:optimization problem0} is equivalent to a constrained optimization problem in these settings, i.e., for $t\in[0,T]$,
        \begin{equation}\label{eq:optimization problem2}
        \begin{aligned}
                         \mu_t & \in  \mathop{\arg\min}\limits_{m \in \mathcal{M}(\mathbb{R}^p)} -pf(t,Y^\mu_t, Z^\mu_t,P^\mu_t,m) + l(t,Y^\mu_t, Z^\mu_t,P^\mu_t,m) + \frac{\sigma^2}{2}Ent(m \mid e^{-U}),
        \end{aligned}
         \end{equation}
with a constraint $\int m (a) da = 1$.

    By Lagrange multiplier method, we further transform this optimization problem into an equivalent optimization problem without constraint. For this, we introduce the Lagrange function $L: \mathcal{M}(\mathbb{R}^p) \times \mathbb{R} \rightarrow \mathbb{R}$ with the Lagrange multiplier $\beta$ as below
         \begin{equation*}\label{eq:lagrange equation}
             L(m, \beta) = H^{\sigma}(t,Y^\mu_t, Z^\mu_t,P^\mu_t,m,e^{-U}) + \beta (\int m (a) da - 1).
         \end{equation*}
            Then we define the Lagrange duality function 
                 $G(\beta) = \min\limits_{m \in \mathcal{M}(\mathbb{R}^p)} L(m, \beta)$.
         By the weak duality theory, we know %that $G(\beta)$ is the lower bound of equivalent original problem in the sense that
         $G(\beta) \leq \min\limits_{m \in \mathcal{P}(\mathbb{R}^p)} H^{\sigma}(t,Y_t, Z_t,P_t,m,e^{-U})$. Hence the goal now is to solve
         \begin{equation}\label{eq:lagrange duality problem}
             \beta^{*} \in \mathop{\arg\max}\limits_{\beta \in \mathbb{R}} G(\beta) = \mathop{\arg\max}\limits_{\beta \in \mathbb{R}} \min_{m \in \mathcal{M}(\mathbb{R}^p)} L(m, \beta).
         \end{equation}
        With Assumption \ref{ass: convex},
        (\ref{eq:optimization problem2}) is a convex optimization problem and satisfies Slater's condition of convex optimization theory, which leads to the strong duality of (\ref{eq:optimization problem2}) and (\ref{eq:lagrange duality problem}):
         \begin{equation*}\label{eq:strong duality}
             \max_{\beta \in \mathbb{R}} \min_{m \in \mathcal{M}(\mathbb{R}^p)} L(m, \beta) = \min_{m \in \mathcal{P}(\mathbb{R}^p)} H^{\sigma}(t,Y^\mu_t, Z^\mu_t, P^\mu_t,m,e^{-U}).
         \end{equation*}
By the first order condition for the flat derivative 
of $H^0$ 
we have
         $$
         \left\{\begin{aligned}
         &\frac{\delta H^{0}}{\delta m}(t,Y^\mu_t, Z^\mu_t, P^\mu_t,m)(a) + \frac{\sigma^2}{2}(U(a) + \ln m(a) + 1) + \beta  = 0,\\
         & \int m (a) da = 1.
         \end{aligned}\right.
         $$
By solving above equation, we know that the control $\mu_t$, $t\in[0,T]$, is a fixed point of the following equation
\begin{equation}\label{eq:fixed point}
         \mu_t(a) = \frac{e^{-U(a) - \frac{2}{\sigma^2}\frac{\delta H^{0}}{\delta m}(t,Y^{\mu}_t, Z^{\mu}_t, P^{\mu}_t,\mu_t)(a)}}{\int e^{-U(a) - \frac{2}{\sigma^2}\frac{\delta H^{0}}{\delta m}(t,Y^{\mu}_t, Z^{\mu}_t, P^{\mu}_t,\mu_t)(a) }da},
         \end{equation}
and the corresponding Lagrange multiplier
    \begin{equation}\label{cz5}
         \begin{aligned}
         \beta = \frac{\sigma^2}{2} \Big(\ln (\int e^{-U(a) - \frac{2}{\sigma^2}\frac{\delta H^{0}}{\delta m}(t,Y^{\mu}_t, Z^{\mu}_t,P^{\mu}_t,\mu_t)(a) }da) -1\Big).
         \end{aligned}
\end{equation}
Note that here $\mu_t\in\mathcal{P}(\mathbb{R}^p)$, so it is not a solution to Problem (\textbf{P1}) unless $\mu_t\in\mathcal{P}_2(\mathbb{R}^p)$ and $\mathbb{E}\left[ \int_0^T Ent(\mu_t\mid e^{-U}) d a d t \right]<\infty$. Actually, according to \cite{hu-2019} we further know from the assumption on $U$ that there exist constants $C^\prime$ and $C$ satisfying $0 \leq C^\prime \leq C$ such that for any $a \in \mathbb{R}^p$,
\begin{equation*}\label{eq:U}
    C^\prime |a|^2 - C \leq U(a) \leq C(1 + |a|^2).
\end{equation*}
So if $\mu$ in (\ref{eq:fixed point}) satisfies $\mathbb{E}\left[ \int_0^T Ent(\mu_t\mid e^{-U}) d a d t \right]<\infty$,
%\begin{equation}\label{eq: finite entropy}
$Ent(\mu_t \mid e^{-U}) < \infty$ a.e. a.s.,
%\end{equation}
which leads to
$
\int |a|^2 \mu_t(a) da \leq \int U(a) \mu_t(a) da < \infty$ a.e. a.s., i.e. $\mu_t \in \mathcal{P}_2(\mathbb{R}^p)$.

Without a specific form of $H^0$, the above discussion for the existence of an optimal control is based on some assumptions and the optimal control remains implicit. We would give an explicit form of $\mu$ in \eqref{eq:fixed point} and prove that this $\mu$ is exactly the optimal control in the LQ case. One can refer to Proposition 2.5 in \cite{hu-2019} for more discussions about similar formulations as (\ref{eq:fixed point}) in deterministic cases.

To end this section, let's see a relaxed optimal control problem of BSDE, which is a special case of entropy regularized control problem. 
For $f: \Omega\times[0,T] \times \mathbb{R}^n \times \mathbb{R}^{n} \times \mathbb{R}^{n} \times \mathbb{R}^p \rightarrow \mathbb{R}^n$, $\xi \in : \Omega\rightarrow \mathbb{R}^n$ and an admissible control $a\in \mathcal{U}_{ad} = L^{2}_{\mathscr{F}}(0,T;\mathbb{R}^p)$, consider the controlled BSDE
\begin{equation}\label{equation:bsde2}
\left\{\begin{aligned}
-d y^{a}_t & =f(t,y^{a}_t, z^{a}_t,a_t) d t-z^{a}_t d W_t, \\
y_t & =\xi.
\end{aligned}\right.
\end{equation}
By law of large numbers we have the exploratory BSDE
\begin{equation}\label{equation:bsde3}
\left\{\begin{aligned}
-d \tilde{y}^\pi_t & =\tilde{f}(t, \tilde{y}^\pi_t, \tilde{z}^\pi_t,  \boldsymbol{\pi}_t) d t-\tilde{z}^\pi_t d W_t, \\
y_t & =\xi ,
\end{aligned}\right.
\end{equation}
where $\boldsymbol{\pi}$ is the distribution of control, $(\tilde{y}^\pi, \tilde{z}^\pi)$ is the exploratory state variable and $\tilde{f}(t, \tilde{y}^\pi_t, \tilde{z}^\pi_t, \boldsymbol{\pi}_t) = \int f(t,\tilde{y}^\pi_t, \tilde{z}^\pi_t, a)\boldsymbol{\pi}_t(a) da$. To see how to get BSDE \eqref{equation:bsde3}, let's set $(y^i, z^i)$ to be the copy of the path generated from the dynamics (\ref{equation:bsde2}) with the control $a^i$ sampled independently under this policy $\boldsymbol{\pi}$.
For any $0\leq t\leq T$, we have
$$
\Delta y_t^i \equiv y_{t+\Delta t}^i-y_t^i \approx -f\left(t,y_t^i,z_t^i, a_t^i\right) \Delta t+z_t^i\left(W_{t+\Delta t}^i-W_t^i\right).
$$
Here each $y^i, i=1,2, \ldots, N$, can be viewed as a copy of an independent sample from $\tilde{y}$. It then follows from  the law of large numbers that, as $N \rightarrow \infty$,
$$
\begin{aligned}
\frac{1}{N} \sum_{i=1}^N \Delta y_t^i&\approx -\frac{1}{N} \sum_{i=1}^N f\left(t,y_t^i,z_t^i, a_t^i\right)\Delta t  +\frac{1}{N} \sum_{i=1}^N z_t^i \left(W_{t+\Delta t}^i-W_t^i\right)\\
& \stackrel{\text { a.s. }}{\longrightarrow}
 \mathbb{E}\left[\int -f(t,\tilde{y}_t, \tilde{z}_t,a)\boldsymbol{\pi}_t(a ) da \Delta t\right]+\mathbb{E}\left[\int  z_t^i\boldsymbol{\pi}_t(a) d a\right]\mathbb{E}\left[W_{t+\Delta t}-W_t\right] \\
&\quad\quad=  \mathbb{E}\left[\int -f(t,\tilde{y}_t, \tilde{z}_t, a)\boldsymbol{\pi}_t(a ) da \Delta t\right].
\end{aligned}
$$
In the above deduction, we have assumed that both $\boldsymbol{\pi}$ and $\tilde{z}$ are identically distributed over $[t, t+\Delta t]$ and independent of the increments of the sample paths of $W$. Then the entropy-regularized cost function appears as
\begin{equation}\label{eq: func bsde2}
    \begin{aligned}
            &J^{\sigma}(T,\xi;\pi) =\mathbb{E}\left[\int_0^T \Big( \int l\left(t, \tilde{y}_t, \tilde{z}_t, a\right)\boldsymbol{\pi}_t(a)da + \frac{\sigma^2}{2} Ent(\boldsymbol{\pi}_t\mid e^{-U}) \Big)d t+\phi\left(\tilde{y}(0)\right) \right].
    \end{aligned}
\end{equation}
By Definition \ref{def: flat derivative},
\begin{equation*}
    \begin{aligned}
    \frac{\tilde{f}(t, \tilde{y}_t, \tilde{z}_t, \boldsymbol{\pi}(a ))}{\delta m} = f(t, \tilde{y}_t, \tilde{z}_t, a)\ \ {\rm and}\ \
    \frac{\tilde{l}(t, \tilde{y}_t, \tilde{z}_t, \boldsymbol{\pi}(a ))}{\delta m} = l(t, \tilde{y}_t, \tilde{z}_t, a).
    \end{aligned}
\end{equation*}
Hence
\begin{equation*}
    \frac{\delta H^{0}}{\delta m}(t,\tilde{y}_t, \tilde{z}_t, \tilde{p}_t, \boldsymbol{\pi}_t)(a) = -\tilde{p}_tf(t,\tilde{y}_t, \tilde{z}_t, a) + l(t,\tilde{y}_t, \tilde{z}_t, a),
\end{equation*}
where $\tilde{p}$ is the solution to the corresponding adjoint equation, and a candidate optimal control 
for the cost functional (\ref{eq: func bsde2}) is
\begin{equation*}\label{eq: optimal relaxed control}
    \begin{aligned}
           \mu_t(a) 
           = \frac{e^{-U(a) - \frac{2}{\sigma^2}[-p_tf(t,\tilde{y}^\mu, \tilde{z}^\mu, a) + l(t,\tilde{y}^\mu, \tilde{z}^\mu, a)]}}{\int e^{-U(a) - \frac{2}{\sigma^2}[-p_tf(t,\tilde{y}^\mu, \tilde{z}^\mu, a) + l(t,\tilde{y}^\mu, \tilde{z}^\mu, a)]}da}.
    \end{aligned}
\end{equation*}

\section{Backward Stochastic Linear-Quadratic Control System with Entropy Regularization}
\label{sec:LQ}

Let $\mathbb{S}^n$ be the set of all $n \times n $ symmetric matrices,  $\mathbb{S}^n_+$ be the set of all $n \times n $ positive semi-definite matrices, $\hat{\mathbb{S}}^n_+$ be the set of all $n \times n $ positive definite matrices, and $\mathbb{I}_n$ be the $n \times n$ identity matrix.  
For  $t\in[0,T]$, $\pi\in\mathcal{A}$, consider a linear controlled BSDE
\begin{equation}\label{eq:linear bsde}
    \left\{\begin{aligned}
        & dY^\pi_t = (A_t Y^\pi_t + B_t \int a \pi_t(a) da + C_t Z^\pi_t )dt + Z^\pi_t dW_t,\\
        & Y_T = \xi
    \end{aligned}\right.
\end{equation}
and its cost functional
{\small\begin{equation}\label{eq: quadratic cost functional}
\begin{aligned}
         J^{\sigma}(\pi) = \frac{1}{2}\mathbb{E}\Big[ \int_o^T \Big(Y_t^{\pi}H_tY^\pi_t + \int aR_ta \pi_t(a) da+ Z_t^{\pi}N_tZ^\pi_t + \sigma^2 Ent(\pi_t | e^{-U})\Big)dt+ Y_0^{\pi}GY^\pi_0  \Big].
\end{aligned}
\end{equation}}
Then LQ problem is

(\textbf{P2}): to find an optimal $\mu \in \mathcal{A}$ such that
\begin{equation*}
    J^{\sigma}(\mu) = \inf_{\pi \in \mathcal{A}} J^{\sigma}(\pi).
\end{equation*}

We give the assumptions for the coefficients of LQ problem.
\begin{ass}\label{H:linear}
    (i) $\xi \in L^2_{\mathscr{F}_T}(\Omega; \mathbb{R}^n)$, $A, C \in L^{\infty}(0,T;\mathbb{R}^{n \times n})$ and $B \in  L^{\infty} \\(0,T;\mathbb{R}^{n \times p})$.\\
    (ii) $H, N\in L^{\infty}(0,T;\mathbb{S}^n_+)$, $R \in L^{\infty}(0,T;\mathbb{S}_+^p)$ and $G \in \mathbb{S}^n_+$.
\end{ass}

Assumption \ref{H:linear} guarantees that linear BSDE (\ref{eq:linear bsde}) has a unique  solution $(Y^\pi, Z^\pi) \\\in S_{\mathscr{F}}^2\left(0, T ; \mathbb{R}^n\right) \times L_{\mathscr{F}}^2\left(0, T ; \mathbb{R}^{n}\right)$.

From Theorem \ref{cz4}, \eqref{eq:fixed point} and \eqref{cz5}, the necessary condition of optimality in LQ case follows.
\begin{theorem}\label{thm:linear smp}
    Under Assumption \ref{H:linear}, if $\mu\in\mathcal{A}$ is an optimal control of Problem (\textbf{P2}) with corresponding optimal state process $(Y^\mu, Z^\mu)$, then
    \begin{equation*}\label{eq:linear adjoint eq}
    \left\{\begin{aligned}
        & dP^\mu_t = -(A_t P^\mu_t + H_t Y^\mu_t )dt -(C_tP^\mu_t + N_t Z^\mu_t) dW_t,\\
        & P^\mu_0 = -G Y^\mu_0,
    \end{aligned}\right.
\end{equation*}
has a unique solution $P^\mu\in L^2_{\mathscr{F}}(0,T;\mathbb{R}^{n})$ such that for any $t \in [0,T]$,
    \begin{equation*}\label{eq:linear optimal control 0 eq}
    \left\{\begin{aligned}
        & P_t^\mu B_t a + \frac{1}{2} aR_t a + \frac{\sigma^2}{2}(U(a) + \ln \mu_t (a) + 1) + \beta = 0,\ \  {\rm for}\ {\rm any}\ a\in \mathbb{R}^p,\\ %\ \forall a \in O\\
        & \int \mu_t (a) da = 1, %, \ {a.e.} \ {a.s.}
    \end{aligned}\right.
\end{equation*}
where $\beta$ is a random variable coming from Lagrange multiplier method. Moreover, $(Y^\mu, Z^\mu, P^\mu, \mu)$ composes a stochastic Hamiltonian system
\begin{equation}\label{eq:Hamiltonian systems}
    \left\{\begin{aligned}
        & dY^\mu_t = (A_t Y^\mu_t + B_t \int_O a \mu_t(a) da + C_t Z^\mu_t)dt + Z^\mu_t dW_t,\\
        & dP^\mu_t = -(A_t P^\mu_t + H_t Y_t )dt -(C_tP^\mu_t + N_t Z^\mu_t) dW_t,\\
        & Y^\mu_T = \xi, \quad P^\mu_0 = -G Y^\mu_0,\\
        & P_t^\mu B_t a + \frac{1}{2} aR_t a + \frac{\sigma^2}{2}(U(a) + \ln \mu_t (a) + 1) + \beta = 0,\ \  {\rm for}\ {\rm any}\ a\in \mathbb{R}^p,\\
        & \int\mu_t (a) da = 1, %, \ {a.e.} \ {a.s.}
    \end{aligned}\right.
\end{equation}
and Hamiltonian system (\ref{eq:Hamiltonian systems}) gives an optimal control
\begin{equation}\label{eq:optimal slq control partial info}
    \left\{\begin{aligned}
    & \mu_t(a)  = \frac{e^{-U(a)-\frac{2}{\sigma^2}({P}^\mu_t B_t a + \frac{1}{2}a R_t a)}}{\int e^{-U(a)-\frac{2}{\sigma^2}({P}^\mu_t B_t a + \frac{1}{2}a^{} R_t a) }da},\\ %\ {a.e.} \ {a.s.}\\
    & \beta = \frac{\sigma^2}{2}(\ln(\int e^{-U(a)-\frac{2}{\sigma^2}({P}^\mu_t B_t a + \frac{1}{2}a R_t a) }da) -1).
    \end{aligned}\right.
\end{equation}
\end{theorem}

We then consider a specific case by setting the reference measure to be a standard normal distribution, i.e. $e^{-U(a)} = \frac{e^{-\frac{|a|^2}{2}}}{\sqrt{(2\pi)^p}}$.
%and $O= \mathbb{R}^p$.
Then since $(R + \frac{\sigma^2}{2}\mathbb{I}_p) \in L^{\infty}(0,T;\hat{\mathbb{S}}^p_+)$, (\ref{eq:optimal slq control partial info}) implies
\begin{equation}\label{eq:linear optimal control}
    \begin{aligned}
            \mu_t(a) & = \frac{e^{-\frac{1}{\sigma^2}(a +(R_t + \frac{\sigma^2}{2}\mathbb{I}_p)^{-1}B_t{P}^\mu_t )^{}(R_t + \frac{\sigma^2}{2}\mathbb{I}_p)(a +(R_t + \frac{\sigma^2}{2}\mathbb{I}_p)^{-1}B_t^{}{P}^\mu_t )}}{\int e^{-\frac{1}{\sigma^2}(a +(R_t + \frac{\sigma^2}{2}\mathbb{I}_p)^{-1}B_t^{}{P}^\mu_t )^{}(R_t + \frac{\sigma^2}{2}\mathbb{I}_p)(a +(R_t + \frac{\sigma^2}{2}\mathbb{I}_p)^{-1}B_t^{}{P}^\mu_t )}da} \\
            &= \frac{1}{\sqrt{(2\pi)^p |det(\Sigma^\mu_t)|}} e^{-\frac{1}{2}(a +(R_t + \frac{\sigma^2}{2}\mathbb{I}_p)^{-1}B_t{P}^\mu_t)(\Sigma^\mu_t)^{-1}(a +(R_t + \frac{\sigma^2}{2}\mathbb{I}_p)^{-1}B_t{P}^\mu_t)},
    \end{aligned}
\end{equation}
where $\Sigma^\mu_t=\frac{\sigma^2}{2}(R_t + \frac{\sigma^2}{2}\mathbb{I}_p)^{-1}$. It appears that $\mu_t$ has a Gaussian distribution and $\Sigma^\mu_t$ is the covariance matrix of $\mu_t$. Hence, from (\ref{eq:linear optimal control}) we know $v_t^\mu:=\int a\mu_t(a)da = -(R_t + \frac{\sigma^2}{2}\mathbb{I}_p)^{-1}B_t{P}^\mu_t$ and $\mu_t\in\mathcal{P}_2(\mathbb{R}^p)$. 

\begin{rmk}
    If we only take $U(\cdot) \equiv 0$ and assume that $ R \in L^{\infty}(0,T;\hat{\mathbb{S}}^p_+)$, then $v_t$ coincides with the strict control, %\cite{wang-2021},
and the optimal control in this case satisfies
$$
\mu_t(a)=  \frac{1}{\sqrt{(2\pi)^p |det(\Sigma^\mu_t)|}} e^{-\frac{1}{2}(a +R_t^{-1}B_t{P}^\mu_t)(\Sigma^\mu_t)^{-1}(a +R_t ^{-1}B_t{P}^\mu_t)},
$$
where the covariance matrix $\Sigma^\mu_t = \frac{\sigma^2}{2} R_t^{-1}\mathbb{I}_p$ and 
$tr(\Sigma^\mu_t R_t) 
=\frac{\sigma^2}{2}p$. Also we can define the cost of exploration (COE) as in \cite{wang-2020},
\begin{align*}%\label{eq:coe}
     COE &:= \frac{1}{2} \mathbb{E}[\int_0^T \int aR_ta\mu_t(a)da - v_t^{*}R_tv_t^\mu dt ] \\
     & = \frac{1}{2} \mathbb{E}[\int_0^T \int (a-v_t^\mu) R_t(a-v_t^\mu)\mu_t(a)da dt ]\\
     & = \frac{1}{2} \mathbb{E}[\int_0^T tr(\Sigma^\mu_tR_t) dt]\\
     & = \frac{1}{2}  \mathbb{E}[\int_0^T\frac{\sigma^2}{2}pdt]\\
     & = \frac{\sigma^2}{4}pT.
\end{align*}
As $\sigma \rightarrow 0$, the cost of relaxed control degenerates to the cost of strict control in the following sense:
\begin{equation*}
    \left\{\begin{aligned}
        &  \mu_t \rightarrow \delta_{v^\mu_t}\ \text{weakly}\ \text{as}\ \sigma \rightarrow  \ {\rm a.e.} \ {\rm a.s.},\\
        & \lim_{\sigma \rightarrow 0} COE = 0,
    \end{aligned}\right.
\end{equation*}
where $\delta_{v^\mu_t}$ stands for the Dirac measure defined at $v^\mu_t$ (see also Exercise 14.4.2 in Klenke \cite{klenke-2020}).
\end{rmk}

Similar to Theorem \ref{thm:verification theorem1}, we give the sufficient condition for an optimal control in LQ case.
\begin{theorem}\label{thm:linear verification thm}
    Under Assumption \ref{H:linear}, %for given $\sigma \geq 0$,
    the control $\mu \in\mathcal{A}$ is an optimal control of Problem (\textbf{P2}) if for any $t\in[0,T]$, it satisfies the Hamiltonian system (\ref{eq:Hamiltonian systems})
   with $(Y^\mu, Z^\mu)$ and $P^\mu$ be the solutions to the corresponding BSDE \eqref{eq:linear bsde} and adjoint equation (\ref{eq:linear adjoint eq}), respectively.
\end{theorem}

Then in the case $e^{-U(a)} = \frac{e^{-\frac{|a|^2}{2}}}{\sqrt{(2\pi)^p}}$, we study the existence and uniqueness of optimal control in LQ case following the decoupling technique for backward stochastic system proposed in Lim and Zhou \cite{lim-2001}.

To begin with, assume that $Y^\mu$ has a decoupling form like
\begin{equation}\label{eq:decoupling}
    Y^\mu_t = \Theta_t {P}^\mu_t + \phi_t,
\end{equation}
where %$\hat{P}_t=\mathbb{E}\left[P_t\mid \mathcal{G}_t\right]$,
$\Theta$ is a deterministic process with $\Theta_T =0$ and differentiable on $t$, and $\phi$ satisfies BSDE
\begin{equation}\label{eq:phi01}
    \left\{\begin{aligned}
        & d \phi_t = \lambda_t dt + \eta_t dW_t,\\
        & \phi_T = \xi
    \end{aligned}\right.
\end{equation}
for some adapted processes $\lambda$ and $\eta$ which will be determined later.

Applying It\^{o} formula to $Y^\mu_t$, by (\ref{eq:Hamiltonian systems}) and (\ref{eq:decoupling}) we have $$
\begin{aligned}
     0 & = d Y^\mu_t - \dot{\Theta}_t{P}^\mu_tdt - \Theta_t d{P}^\mu_t - d\phi_t\\
       & = \left(A_tY^\mu_t + B_t \int a \mu_t(a) da + C_tZ^\mu_t \right)dt + Z^\mu_tdW_t\\
       & \quad \ \ - \dot{\Theta}_t{P}^\mu_tdt + \Theta_t(A_t{P}^\mu_t + H_t{Y}^\mu_t)dt + \Theta_t(C_t{P}^\mu_t + N_t{Z}^\mu_t)dW_t- \lambda_t dt - \eta_t dW_t.
\end{aligned}
$$
Bearing in mind that in this case $\int a \mu_t(a) da = -(R_t + \frac{\sigma^2}{2}\mathbb{I}_p)^{-1}B_t{P}^\mu_t$, we further have
$$
\left\{\begin{aligned}
     &  \lambda_t = A_tY^\mu_t - B_t(R_t + \frac{\sigma^2}{2}\mathbb{I}_p)^{-1}B_t{P}^\mu_t + C_tZ_t  - \dot{\Theta}_t{P}^\mu_t + \Theta_t(A^{}_t{P}^\mu_t + H_t {Y}^\mu_t),\\
     & Z^\mu_t + \Theta_t(C_t{P}^\mu_t + N_t{Z}^\mu_t) - \eta_t = 0,\\
     %& \bar{Z}_t - \bar{\eta}_t = 0,
\end{aligned}\right.
$$
which implies $Z^\mu_t = (\mathbb{I}_p + \Theta_t N_t)^{-1}({\eta}_t - \Theta_t C_t {P}^\mu_t)$.
Since the coefficient ahead of ${P}^\mu_t$ is $0$, we get the Riccati equation
\begin{equation}\label{eq:riccati eq}
        \left\{\begin{aligned}
        & \dot{\Theta}_t - A_t\Theta_t - \Theta_t A_t - \Theta_t H_t \Theta_t + (R_t + \frac{\sigma^2}{2}\mathbb{I}_p)^{-1}B_t + C_t (\mathbb{I}_p + \Theta_t N_t)^{-1} \Theta_t C^{}_t = 0,\\
        & \Theta_T = 0.
    \end{aligned}\right.
\end{equation}
It is well known that the above Riccati equation (\ref{eq:riccati eq}) has a unique solution $\Theta \in L^{\infty}(0,T; \mathbb{S}^n_+)$ (see e.g. \cite{lim-2001}). Hence \eqref{eq:phi01} can be rewritten as below:
\begin{equation}\label{eq:phi02}
    \left\{\begin{aligned}
        & d \phi_t = \left((A_t + \Theta_t H_t){\phi}_t +  C_t (\mathbb{I}_p + \Theta_t N_t)^{-1}{\eta}_t \right) dt + \eta_t dW_t,\\
        & \phi_T = \xi.
    \end{aligned}\right.
\end{equation}
The solvability of BSDE (\ref{eq:phi02}) comes from the classical results in \cite{par-pen1}.

\begin{theorem}
    If the reference measure is a standard normal distribution, i.e. $e^{-U(a)} = \frac{e^{-\frac{|a|^2}{2}}}{\sqrt{(2\pi)^p}}$, under Assumption \ref{H:linear}, stochastic Hamiltonian system (\ref{eq:Hamiltonian systems}) has a unique solution $(Y^\mu,Z^\mu,P^\mu, \mu)$, where for any $t \in [0,T]$,
 \begin{equation}\label{cz7}
         \left\{\begin{array}{l}
              Y^\mu_t = \Theta_t{P}^\mu_t + \phi_t , \\
              Z^\mu_t = (\mathbb{I}_n + \Theta_t N_t)^{-1}({\eta}_t - \Theta_t C_t {P}^\mu_t), \\
              %\bar{Z}_t = \bar{\eta}_t,\\
              Y^\mu_0 = (\mathbb{I}_n + \Theta_0 G)^{-1}\phi_0,
         \end{array}\right.
\end{equation}
$\mu_t$ is Gaussian with the covariance matrix $\Sigma^\mu_t=\frac{\sigma^2}{2}(R_t + \frac{\sigma^2}{2}\mathbb{I}_p)^{-1}$ and the mean $v_t^\mu$ satisfying
    $
    (R_t + \frac{\sigma^2}{2}\mathbb{I}_p)v^\mu_t + B_tP^\mu_t = 0$,
    $\Theta$ is the solution to Riccati equation (\ref{eq:riccati eq}) and $(\phi, \eta)$ is the solution to BSDE (\ref{eq:phi02}).
\end{theorem}
\begin{proof}
We first verify that (\ref{cz7}) and the Gaussian random variable $\mu_t$ give a solution to stochastic Hamiltonian systems (\ref{eq:Hamiltonian systems}). Consider SDE
\begin{equation}\label{eq: sde with filtering}
        \left\{\begin{array}{l}
          d P^\mu_t = - \left(A_t P^\mu_t + H_t(\Theta_t {P}^\mu_t + \phi_t)\right)dt\\
          \quad \quad \quad - \left(C_t P^\mu_t + N_t \left( (\mathbb{I}_n + \Theta_t N_t)^{-1}({\eta}_t - \Theta_t C_t {P}^\mu_t)\right)\right)d W_t,\\
          %\quad \quad \quad - \left(\bar{C}_t P_t + \bar{N}_t \bar{\eta}_t\right) d \bar{W}_t,\\
          P^\mu_0 = - G  (\mathbb{I}_n + \Theta_0 G)^{-1}\phi_0.
         \end{array}\right.
\end{equation}
Obviously, the linear SDE \eqref{eq: sde with filtering} has a unique solution $P^\mu$.
Applying It\^o formula to $Y^\mu_t = \Theta_t {P}^\mu_t + \phi_t$, we have
\begin{align*}
d Y^\mu_t= &\left(A_t Y^\mu_t - B_t(R_t + \frac{\sigma^2}{2}\mathbb{I}_p)^{-1}B_t{P}^\mu_t + C_t (\mathbb{I}_n + \Theta_t N_t)^{-1}({\eta}_t - \Theta_t C_t{P}^\mu_t)\right)dt \\
        & +(\mathbb{I}_n + \Theta_t N_t)^{-1}({\eta}_t - \Theta_t C_t {P}^\mu_t) dW_t,
    \end{align*}
     and $Y^\mu_0 = (\mathbb{I}_n + \Theta_0 G)^{-1}\phi_0$. Noticing
             $Z^\mu_t =  (\mathbb{I}_n + \Theta_t N_t)^{-1}({\eta}_t - \Theta_t C_t {P}^\mu_t)$,
    we know that $(Y^\mu,Z^\mu,P^\mu)$ satisfies (\ref{eq:Hamiltonian systems}). As for $\mu$, its explicit form (\ref{eq:linear optimal control}) deduced from (\ref{eq:Hamiltonian systems}) and the argument below it demonstrate that $\mu_t$ satisfies a Gaussian distribution with the covariance matrix $\Sigma^\mu_t=\frac{\sigma^2}{2}(R_t + \frac{\sigma^2}{2}\mathbb{I}_p)^{-1}$ and the mean $v_t^\mu$ satisfying
    $
   (R_t + \frac{\sigma^2}{2}\mathbb{I}_p)v^\mu_t + B_tP^\mu_t  = 0$.
So we prove that $(Y^\mu,Z^\mu,P^\mu, \mu)$ is a solution to stochastic Hamiltonian system (\ref{eq:Hamiltonian systems}).

As for the uniqueness of optimal control, assume that stochastic Hamiltonian systems (\ref{eq:Hamiltonian systems}) has two solutions $(Y, Z, P, \mu)$ and $(Y^{\prime}, Z^{\prime}, P^{\prime}, \mu^{\prime})$. Let $\bar{\varphi} = \varphi -\varphi^\prime$, $\varphi={Y},{Z},{P},{\mu}$ and $\bar{v}= \int a\bar{\mu}(a)da$. Then $(\bar{Y},\bar{Z},\bar{P},\bar{v})$ satisfies
$$
\left\{\begin{array}{l}
d \bar{Y}_t=\left(A_t \bar{Y}_t+B_t \bar{v}_t+C_t \bar{Z}_t\right) d t+\bar{Z}_t d W_t, \\
d \bar{P}=-\left(A_t \bar{P}_t+H_t \bar{Y}_t\right) d t-\left(C_t \bar{P}_t+N_t \bar{Z}_t\right) d W_t, \\
\bar{Y}_T=0,\ \ \  \bar{P}_0=-G \bar{Y}_0, \\
(R_t + \frac{\sigma^2}{2}\mathbb{I}_p) \bar{v}_t+B_t \bar{P}_t =0.
\end{array}\right.
$$
Applying It\^{o} formula to $\bar{Y}_t \bar{P}_t$, we obtain
$$
\begin{aligned}
\mathbb{E}\left[\bar{Y}_0 G \bar{Y}_0\right]
=-\mathbb{E}\left[\int_0^T\left(\bar{Y}_t H_t\bar{Y}_t+{\bar{P}}^{}_t B_t (R_t + \frac{\sigma^2}{2}\mathbb{I}_p)^{-1} B^{}_t \bar{P}_t+\bar{Z}^{}_t N_t \bar{Z}_t\right) d t\right].
\end{aligned}
$$
From Assumption \ref{H:linear}, we know $G, H, N\in\mathbb{S}^n_+$ and $(R_t + \frac{\sigma^2}{2}\mathbb{I}_p)\in\hat{\mathbb{S}}^n_+$. Hence
$
B_t \bar{P}_t=0
$ a.s.
Consequently, $\left(\bar{Y}, \bar{Z}\right)$ satisfies
\begin{equation}\label{cz3}
\left\{\begin{array}{l}
d \bar{Y}_t=\left(A_t \bar{Y}_t+C_t \bar{Z}_t\right) d t+\bar{Z}_t d W_t, \\
\bar{Y}_T=0.
\end{array}\right.
\end{equation}
Obviously, (\ref{cz3}) has a unique solution $\left(\bar{Y}, \bar{Z}\right) \equiv 0$. So
$$
\left\{\begin{array}{l}
d \bar{P}=-A_t \bar{P}_t d t-C_t \bar{P}_t d W,\\
\bar{P}_0=-G \bar{Y}_0,
\end{array}\right.
$$
also suggests $\bar{P} \equiv 0$, and then $\bar{\mu} \equiv 0$ follows immediately from the means and covariances of the optimal controls are identical.
\end{proof}

Moreover, for a suitable reference measure, the solution to stochastic Hamiltonian system (\ref{eq:Hamiltonian systems}) is the optimal control of backward stochastic LQ control system with entropy regularization.
\begin{corollary}
    If the reference measure is a standard normal distribution, i.e. $e^{-U(a)} = \frac{e^{-\frac{|a|^2}{2}}}{\sqrt{(2\pi)^p}}$, under Assumption \ref{H:linear}, the solution to stochastic Hamiltonian system (\ref{eq:Hamiltonian systems}) is the unique optimal control of Problem (\textbf{P2}).
\end{corollary}
\begin{proof}
We only need to prove that the solution $\mu$ to stochastic Hamiltonian system (\ref{eq:Hamiltonian systems}) is an admissible control of Problem (\textbf{P2}).

Recall that $\mu_t$ is Gaussian with the covariance matrix $\Sigma^\mu_t=\frac{\sigma^2}{2}(R_t + \frac{\sigma^2}{2}\mathbb{I}_p)^{-1}$ and the mean $v_t^\mu$ satisfying
    $
    (R_t + \frac{\sigma^2}{2}\mathbb{I}_p)v^\mu_t + B_tP^\mu_t = 0$. Noticing $(R + \frac{\sigma^2}{2}\mathbb{I}_p) \in L^{\infty}(0,T;\hat{\mathbb{S}}^p_+)$, $B \in L^{\infty}(0,T;\mathbb{R}^{n \times p})$ and $P^\mu \in S^2_{\mathscr{F}}(0,T;\mathbb{R}^{n})$, we first have
{\small\begin{align*}
    & \mathbb{E}\left[\int_0^T \int|a|^2 \mu_t(a) d a d t\right] = \mathbb{E}\left[\int_0^T\left(|\nu^\mu_t|^2+ tr(\Sigma^\mu_t)\right)dt\right]\leq C\left( 1 + \mathbb{E}\left[\int_0^T |P^\mu_t|^2dt\right]\right)< \infty.
\end{align*}}
On the other hand,
\begin{align*}
    \mathbb{E}\left[ \int_0^T Ent(\mu_t\mid e^{-U}) d a d t\right]&= \mathbb{E}\left[ \int_0^T \frac{1}{2}\left(-\ln(det (\Sigma^\mu_t)) + tr(\Sigma^\mu_t)+ |\nu^\mu_t|^2- p\right)dt\right]\\
    & \leq C\left( 1 + \mathbb{E}\left[\int_0^T|P^\mu_t|^2 dt\right]\right)<\infty.
\end{align*}
Therefore, $\mu \in\ \mathcal{A}$ follows.
\end{proof}

\bibliographystyle{siam}
\bibliography{references}

\end{document}